\documentclass[11pt]{amsart}
\usepackage[colorlinks=true,pagebackref,hyperindex,citecolor=green,linkcolor=blue]{hyperref}

\usepackage{amsmath}
\usepackage{amsfonts}
\usepackage{amssymb}
\usepackage{amsthm}
\usepackage{stmaryrd}
\usepackage[all]{xy}
\usepackage{mathrsfs}
\usepackage{enumitem} 
\usepackage[T1]{fontenc}
\usepackage{color}
\usepackage[table]{xcolor}

\usepackage{setspace}
\usepackage{moreverb}
\usepackage{url}
\usepackage[all]{xy}
\usepackage{lscape}
\usepackage{tikz}
\usetikzlibrary{arrows}
\usepackage{color}

\newcommand{\ignore}[1]{}

\newtheorem{theorem}{Theorem}[section]
\newtheorem{lemma}[theorem]{Lemma}
\newtheorem{corollary}[theorem]{Corollary}
\newtheorem{proposition}[theorem]{Proposition}

\theoremstyle{definition}
\newtheorem{definition}[theorem]{Definition}
\newtheorem{example}[theorem]{Example}

\theoremstyle{remark}
\newtheorem{remark}[theorem]{Remark}

\newtheorem{assumptions}[theorem]{Assumptions}

\newtheorem{discussion}[theorem]{Discussion}

\theoremstyle{notation}

\numberwithin{equation}{section}

\input xy
\xyoption{all}

\definecolor{ncs}{rgb}{1.0, 0.83, 0.0}
\definecolor{usc}{rgb}{1.0, 0.7, 0.0}

\newcommand{\bZ}{{\mathbb{Z}}}

\definecolor{grey}{rgb}{0.75,0.75,0.75}
\definecolor{orange}{rgb}{1.0,0.5,0.5}
\definecolor{brown}{rgb}{0.5,0.25,0.0}
\definecolor{pink}{rgb}{1.0,0.5,0.5}

\newcommand{\hlt}{\mathrm {ht\, }}

\newcommand{\dm}{\mathrm {dim }}

\newcommand{\Supp}{\mbox{\rm{Supp}} }

\newcommand{\fM}{{\mathfrak m}}

\newcommand{\fp}{\mathfrak{p}}

\newcommand{\cP}{{\mathcal P}}

\newcommand{\la}{{\lambda}}

\newcommand{\lra}{{\longrightarrow}}

\newcommand{\E}[1]{\mbox{E}_{R}(R/#1)}

\newcommand{\K}{\mathbb{K}}

\newcommand{\m}{\mathfrak{m}}

\begin{document}

\title[Bass numbers of local cohomology of cover ideals of graphs]{Bass numbers of local cohomology of cover ideals of graphs}

\author[J. \`Alvarez Montaner]{Josep \`Alvarez Montaner}
\address{Departament de Matem\`atiques\\
Universitat Polit\`ecnica de Catalunya\\ Av. Diagonal 647,
Barcelona 08028, SPAIN} \email{Josep.Alvarez@upc.edu}

\author[F. Sohrabi]{Fatemeh Sohrabi}

\address{Faculty of Sciences, Departament of Mathematics\\
University of Mohagheh Ardabili\\56199-11367 Ardabil, Iran} \email{fateme\_sohrabi@uma.ac.ir}

\thanks{The first author was partially 
supported by Generalitat de Catalunya 2017SGR-932 project
and Spanish Ministerio de Econom\'ia y Competitividad
MTM2015-69135-P. He is also a member of the Barcelona Graduate School of Mathematics (BGSMath, MDM-2014-0445) 
}



\newcommand{\ba}{\mathbf{a}}
\newcommand{\bb}{\mathbf{b}}
\newcommand{\bc}{\mathbf{c}}
\newcommand{\be}{\mathbf{e}}
\newcommand{\gmod}{\operatorname{* mod}}
\newcommand{\Sq}{\operatorname{Sq}}
\newcommand{\relint}{\operatorname{rel-int}}

\newcommand{\KK}{\mathbb{K}}
\newcommand\const{\underline{\K}}
\newcommand\Dcom{\mathcal{D}^\bullet}
\newcommand\cExt{{\mathcal Ext}}
\newcommand\Db{{\mathsf D}^b}
\newcommand\GG{\mathbb G}
\newcommand{\bA}{\mathbf{A}}
\newcommand{\EE}{\mathbb{E}}
\newcommand\LL{\mathbb L}

\begin{abstract}
We develop splitting techniques to study Lyubeznik numbers of cover ideals of graphs which allow us to describe them for large families of graphs including forests, cycles, wheels and cactus graphs. More generally we are able to compute all the Bass numbers and the shape of the injective resolution of local cohomology modules by considering the connected components of the corresponding subgraphs.  Indeed our method gives us a very simple criterion for the vanishing of these local cohomology modules in terms of the connected components.
\end{abstract}

\maketitle

\section{Introduction}
Let $G=(V_G,E_G)$ be a finite graph in the set of vertices $V_G=\{x_1,\dots , x_n\}$ and the set of edges $E_G$
We will assume that the graph is simple so no multiple edges between vertices or loops are allowed. In order to study such a graph from an algebraic point of view one may associate a monomial ideal in the polynomial ring $R=\K[x_1,\dots,x_n]$, where $\K$ is a field. There are several ways to do so but the most used in the literature are:

$\cdot$ {\bf Edge ideal:} $I(G)=(x_i x_j \hskip 2mm | \hskip 2mm \{x_{i},x_{j}\}\in E_G)$.

$\cdot$ {\bf Cover ideal:} $J(G)=\underset{\{x_{i},x_{j}\}\in E_G}{\bigcap} (x_i,x_j)$.

\noindent That is, the edges of the graph describe the generators of the edge ideal and the primary components of the cover ideal. 
A common theme in combinatorial commutative algebra has been to understand graph theoretic properties of $G$ from the algebraic properties of the associated ideal and vice versa.

For instance, a lot of attention has been paid to the study of free resolutions of edge ideals and its associated invariants such as Betti numbers, projective dimension or Castelnuovo-Mumford regularity. Notice that edge ideals are a very particular class of squarefree monomial ideals so one can use some of the techniques available in this context such as Hochster's formula \cite{Hoc}, splitting techniques \cite{EK, FHV} or discrete Morse theory \cite{BW}. 
Despite these efforts, a full description of these algebraic invariants is only known for a few families of graphs. 

The aim of this paper is to study Bass numbers of local cohomology modules supported on cover ideals of graphs. 
The choice of cover ideals instead of edge ideals is because free resolutions and local cohomology modules of any squarefree monomial ideal are related via Alexander duality (see \cite{Mu00}, \cite{Mil_thesis},\cite{Mi00}, \cite{AV14}) and, in the particular case we are considering,  edge and cover ideals are Alexander dual to each other. Moreover it seems more natural to use the primary decomposition of an ideal if we want to use the Mayer-Vietoris sequence to study local cohomology modules.

In order to compute the Bass numbers of local cohomology modules of any squarefree monomial ideal we may refer to the work of K.~Yanagawa \cite{Ya01}  or the work of the first author with his collaborators in \cite{Al00}, \cite{Al04_2}, \cite{AV14}. Indeed,  one can use the computational algebra system {\tt Macaulay 2} \cite{GS} to compute them as it has been shown in \cite{AF13}. We point out that, using the restriction functor, we may just  reduce to the case of studying Bass numbers with respect to the homogeneous maximal ideal, which are also known as Lyubeznik numbers \cite{Ly93}.

The methods presented in \cite{AV14} may seem quite appropriate  for the case of cover ideals of graphs. Namely, in order to compute  Lyubeznik numbers, one has to describe the linear strands of the free resolution of the corresponding edge ideal and compute the homology groups of a complex of $\KK$-vector spaces associated to these linear strands. However, even though one may find some explicit free resolutions of edge ideals of graphs in the literature, it seems quite complicated to give closed formulas for the Lyubeznik numbers even for simple families of graphs. 

In this paper we shift gears and we present some splitting techniques that would allow us to compute the Lyubeznik numbers of large families of graphs without any previous description of its local cohomology modules or equivalently, the free resolution of the corresponding edge ideals. 
The idea behind these splitting techniques is to relate the Lyubeznik numbers of our initial graph to the Lyubeznik numbers of the subgraph obtained by removing a vertex. Indeed, the Lyubeznik table remains invariant when we remove a whisker or even a $3$ or $4$-cycle. Moreover we can control the Lyubeznik table when we remove degree two vertices or a dominating vertex. To compute all the Lyubeznik numbers of any given graph in a fixed number of vertices is out of the scope of this work but  we can reduce enormously the number of cases that we have to consider by a simple inspection of the shape of the graph. 
More generally, we can compute all the Bass numbers of local cohomology modules just considering subgraphs of our initial graph. In particular we can describe the linear strands of the injective resolution of these modules. The structure of these injective resolutions depend on the number 
of connected components of the corresponding subgraphs.  Quite nicely, we deduce a vanishing criterion  for local cohomology modules depending of these connected components of the subgraphs.

We should mention that, in general, Lyubeznik numbers depend on the characteristic of the base field. However, all the methods we develop here are independent of the characteristic, meaning that the Lyubeznik numbers of a graph will depend on the characteristic if and only if the Lyubeznik numbers of the graph obtained after removing a vertex also depend on the characteristic.

The organization of this paper is as follows. In Section \ref{local_cohomology} we introduce all the basics on local cohomology supported on squarefree monomial ideals and its injective resolution.  In particular we introduce Bass numbers and how to describe them using the graded pieces of the composition of local cohomology modules. Since we can always reduce to the case of Lyubeznik numbers we briefly recall  in Subsection \ref{Lyubeznik1} its definition and the main properties we are going to use throughout this work.  In Subsection \ref{free} we review the relation between Lyubeznik numbers and linear strands of the Alexander dual ideal. Finally, in Subsection \ref{split1} we propose the notion of MV-splitting (see Definition \ref{def_MV})  together with an application of the long exact sequence of local cohomology modules (see Discussion \ref{MVsplit}) that will be crucial later on.
The reason of working in the general framework of squarefree monomial ideals is that, even though we want to study cover ideals of graphs, we will have to leave this context when applying these techniques. Moreover, all these splitting methods could be also applied for any squarefree monomial ideal.

In Section \ref{Lyubeznik} we focus on the study of Lyubeznik numbers of cover ideals of graphs.  Our first result is Theorem \ref{main} in which we describe the Lyubeznik table associated to the cover ideal of a simple connected graph for which the MV-splitting satisfies some extra conditions. 
These conditions are naturally satisfied when we consider splitting vertices and thus we specialize to this case. In Proposition \ref{whisker} prove that the Lyubeznik table remains invariant after removing vertices of degree one.  In  Proposition \ref{handle} we prove that the Lyubeznik table is also invariant if we remove a handle, which is a $3$ or $4$-cycle having a degree two vertex. More generally, we describe in Proposition \ref{deg_two} and Corollary \ref{deg_two_2} the Lyubeznik table of any graph as long as we find degree two splitting vertices. In this way we can apply recursion to reduce the computation to the case of a smaller graph. 

In Subsection \ref{Lyubeznik3} we apply these splitting techniques to compute the Lyubeznik table of some families of examples. We prove that trees have trivial Lyubeznik table and we deduce a formula for the case of forests. Any cone of a graph, for example a wheel,  also has trivial Lyubeznik table. The results on degree two vertices allow us to compute the case of cycles and, more generally, the family of graphs obtained by joining cycles in such a way that we can still find degree two vertices that we can remove in order to simplify the graph. Indeed, after removing whiskers and handles we may  consider cycles joined by paths or sharing edges. Such an example would be the case of cactus graphs or cycles with chords.

In Section \ref{Bass} we study all the Bass numbers of the cover ideal of a graph  by considering the Lyubeznik numbers of the corresponding subgraphs.
In Subsection \ref{Bass1}  we pay attention to a class of graphs (that include forests and Cohen-Macaulay graphs)  whose local cohomology modules have a linear injective resolution. In particular we give a closed formula for these Bass numbers in Theorem \ref{Bass_trivial}. Quite surprisingly we provide in Proposition \ref{max} a vanishing criterion for the local cohomology modules in terms of the number of connected components of the subgraphs. Using Alexander duality it also gives a formula for the projective dimension of the edge ideal of such a graph.
We also study the injective resolution of local cohomology modules of graphs obtained by joining cycles in Subsection \ref{Bass2}.
Finally, we also provide a vanishing criterion for the local cohomology modules associated to the corresponding subgraphs in Proposition \ref{vanishing_cycle} and Proposition \ref{vanishing_cycle2}.

\section{Bass numbers of local cohomology modules} \label{local_cohomology}
Throughout this section we will assume the general framework of a squarefree monomial ideals in a polynomial ring $R=\K[x_1,\dots,x_n]$ with coefficients over a field $\K$. Namely, a squarefree monomial ideal $J\subseteq R$ is generated by monomials of the form ${\bf x^{\alpha}}:= x_1^{a_1}\cdots
x_n^{a_n},\hskip 2mm {\rm where}\hskip 2mm {\bf \alpha}=(a_1,\dots , a_n) \in
\{0,1\}^n$. Its minimal primary decomposition is given in terms
of { face ideals} ${\fp_{\alpha}}:= \langle x_i\hskip 2mm | \hskip
2mm a_i \neq 0 \rangle, \hskip 2mm  {\bf \alpha}\in \{0,1\}^n.$
For simplicity we will denote the homogeneous maximal
ideal $\fM:=\fp_{{\bf 1}}=(x_1,\dots,x_n)$, where ${\bf 1}=(1,\dots,1)$. As usual,  we denote $|\alpha|= a_1
+\cdots+a_n$ and $\varepsilon_1,\dots, \varepsilon_n$ will be
the standard basis of $\bZ^n$. The {\bf Alexander dual} of the ideal $J$ is the squarefree
monomial ideal $J^{\vee}\subseteq R$ defined as $J^{\vee}=({\bf x^{\alpha_1}},\dots , {\bf x^{\alpha_s}})$ associated to the minimal primary decomposition $J= \fp_{\alpha_1}\cap \cdots \cap \fp_{\alpha_s}$.

\vskip 2mm 

Let $\bZ^{\alpha}\subseteq \bZ^n$ be the coordinate space spanned by
$\{\varepsilon_i \hskip 2mm | \hskip 2mm a_i=1\}$, $\alpha\in \{0,1\}^n$.
The {\bf restriction} of $R$ to the face ideal $\fp_\alpha\subseteq R$ is
the $\bZ^{\alpha}$-graded $\K$-subalgebra of $R$ $$R_{\fp_\alpha}:=\K[x_i \hskip 2mm | \hskip 2mm
a_i=1].$$ 


\vskip 2mm

Let $J= \fp_{\alpha_1}\cap \cdots \cap \fp_{\alpha_s}$ be the minimal primary decomposition of a squarefree
monomial ideal $J\subseteq R$.  Then, the restriction of $J$ to the
face ideal $\fp_\alpha$ is the squarefree monomial ideal $$
J_{\fp_\alpha}=\bigcap_{\alpha_j \leq \alpha} \fp_{\alpha_j}\subseteq R_{\fp_\alpha}.$$
Moreover, the restriction of a local cohomology module is
$$[H_J^r(R)]_{\fp_\alpha} = H_{J_{\fp_\alpha}}^r(R_{\fp_\alpha})$$
Roughly speaking, restriction gives us a functor that plays the role
of the localization functor. For details and further considerations we refer to \cite{Mi00}.

\vskip 2mm 

A key fact in its study is that 
local cohomology modules $H_J^r(R)$ supported on monomial ideals are
$\bZ^n$-graded modules. Indeed, these modules satisfy some nice properties since they
fit, modulo a shifting by ${\bf 1}$, into the category of {\bf straight}
(resp. {\bf 1}-determined) modules introduced by K.~Yanagawa
\cite{Ya01} (resp. E.~Miller \cite{Mi00}). In what follows we are going to introduce the basic notions that we are going
to use in this work. Most of them can be found in textbooks such as \cite{BH} and \cite{MS05} or the lecture notes \cite{Alv13}.

\vskip 2mm

In order to give a module structure to the straight
module $H_J^r(R)$ 
we have to describe:

\begin{itemize}
\item[$\cdot$] The graded pieces
$[H_J^r(R)]_{-\alpha}$ for all $\alpha \in \{0,1\}^n$.
\item[$\cdot$] The multiplication morphisms: \hskip 2mm  $\cdot x_i:
[H_J^r(R)]_{-\alpha} \longrightarrow [H_J^r(R)]_{-(\alpha -
\varepsilon_i) }.$
\end{itemize}

\noindent This structure has been described by N.~Terai \cite{Te} and M.~Musta\c{t}\u{a}
\cite{Mu00} in terms of some simplicial complexes associated to the monomial ideal $J$.  The approach considered in \cite{AGZ03}
gives an interpretation in terms of the components appearing in the minimal primary decomposition of $J$ which will be more convenient for our purposes. 

Let $\cP_J$ be the partially ordered set consisting of the sums of ideals in the minimal primary decomposition of $J$ ordered by reverse inclusion. Namely, if $J= \fp_{\alpha_1}\cap \cdots \cap \fp_{\alpha_s}$ is the minimal primary decomposition we have that any ideal $J_p \in \cP_J$ is a certain sum
$J_p= \fp_{\alpha_{i_1}}+ \cdots +\fp_{\alpha_{i_j}}$ and, since the sum of face ideals is a face ideal we have that $J_p=\fp_\alpha$ for some $\alpha \in \{0,1\}^n$.
In what follows we will just denote by $\fp_\alpha$, or simply $\alpha$, the elements of $\cP_J$.

 Let $1_{\cP_J}$ be a terminal element that we add to the poset. To any $\alpha \in \cP_J$ we may consider the {\bf order complex} associated to the  subposet
$(\alpha, 1_{{\cP_J}}):=\{z\in \cP_J \mid\quad \alpha < z < 1_{\cP}\}$ and the dimensions of the reduced simplicial homology  groups 
$$m_{r,\alpha}:=\dim_{\K}\widetilde{H}_{|\alpha|-r-1} ((\alpha,1_{{\cP_J}}); \K). $$
Then, the graded pieces of the local cohomology modules of $J$ can be described as follows:
\begin{equation}\label{pieces}
[H_J^r(R)]_{-\alpha}= \bigoplus_{\alpha \in \cP_J} [H_{\fp_\alpha}^{|\alpha|}(R)^{m_{r,\alpha}}]_{-\alpha}
\end{equation}

The category of straight modules is a category with enough injective modules. Indeed, the indecomposable injective objects are the shifted injective envelopes 
$E_\alpha:={^{\ast}\E{\fp_{\alpha}}({\bf 1})}$, $\alpha \in \{0,1\}^n$, and every graded injective module is isomorphic to a unique (up to order) direct sum of indecomposable injectives.
It follows that the {\bf  minimal $\bZ^n$-graded injective resolution} of a local cohomology module $H_J^r(R)$,
is an exact sequence:
$$\mathbb{I}_{\bullet}(H_J^r(R)): \hskip 3mm \xymatrix{ 0 \ar[r]& H_{J}^{r}(R) \ar[r] & I_{0}
\ar[r]^{d^{0}}&  I_1 \ar[r]^{d^1}&\cdots \ar[r]& I_{m} \ar[r]^{d^m}&
0},$$ where the $p$-th term is $$I_p = \bigoplus_{\alpha \in
{\mathbb{Z}^n}}
E_\alpha^{\mu_{p}(\fp_\alpha,H_J^r(R))}$$ and the invariants defined by ${\mu_{p}(\fp_\alpha,H_J^r(R))}$ are the {\bf Bass numbers} of $H_J^r(R)$.  Given an integer $\ell$, the {\bf
$\ell$-linear strand} of $\mathbb{I}_{\bullet}(H_J^r(R))$ is the complex:
$$\mathbb{I}_{\bullet}^{<\ell>}(H_J^r(R)): \hskip 3mm \xymatrix{ 0 \ar[r]&
I_{0}^{<\ell>} \ar[r]& I_{1}^{<\ell>} \ar[r]& \cdots
\ar[r]& I_{m}^{<\ell>} \ar[r]& 0},$$ where $$I_p^{<\ell>} =
\bigoplus_{|\alpha|=p+\ell} E_\alpha^{\mu_{p}(\fp_\alpha, H_J^r(R))},$$

\begin{remark}
The ($\bZ^n$-graded) Bass numbers coincide with the usual Bass numbers in the
minimal injective resolution of  $H_J^r(R)$ as it was proved by S.~Goto and K.~I.~Watanabe in \cite{GW78}. Indeed they provided a method to compute 
the Bass numbers with respect to any prime ideal. Namely, given any prime ideal $\fp \in
{\rm Spec} R$, let $\fp_\alpha$ be the largest face ideal contained in
$\fp$. If $\hlt(\fp/\fp_\alpha)=s$ then $\mu_p(\fp_\alpha,H_J^r(R))=
\mu_{p+s}(\fp,H_J^r(R))$.

\end{remark}


 
%
 
 \vskip 2mm
 
 The Bass numbers of straight modules, and local cohomology modules in particular, were already studied and described in \cite{Ya01}. The approach that we will use in this work is using the graded pieces of the composition of local cohomology modules. Namely, using  \cite[Corollary 3.6]{AV14} (see also \cite{Al05}), we have:

 \begin{proposition}
 Let $J\subseteq R$ be a squarefree monomial ideal and $\fp_\alpha \subseteq R$ be a face ideal, $\alpha\in\{0,1\}^n$. Then, the Bass numbers of the local cohomology module $H_J^r(R)$ with respect to $\fp_\alpha$ are
 $$\mu_p(\fp_\alpha, H_J^r(R))= \dim_{\K} [H^p_{\fp_\alpha}(H_J^r(R))]_{-\alpha}.$$ In particular, the Bass numbers with respect to the homogeneous maximal ideal $\fM$ are
 $$\mu_p(\fM_\alpha, H_J^r(R))= \dim_{\K} [H^p_\fM(H_J^r(R))]_{-\bf 1}.$$ 
 \end{proposition}

 \begin{remark} \label{Bass_restriction}
  Bass numbers behave well with respect to the restriction functor so we may always assume that the face ideal $\fp_\alpha$ is the maximal ideal. Namely, we have 
$$\mu_p(\fp_\alpha, H_J^r(R))=\mu_p(\fp_\alpha R_{\fp_\alpha}, [H_J^r(R)]_{\fp_\alpha}).$$
 \end{remark}

\subsection{Lyubeznik numbers} \label{Lyubeznik1}
In the seminal works of C.~Huneke and R.~Y.~Sharp \cite{HS93} and G.~Lyubeznik \cite{Ly93} it is proven that
the Bass numbers of local cohomology modules are all finite. This prompted G.~Lyubeznik  to introduce a new set 
of invariants defined as follows:

\vskip 2mm

Let $A$ be a noetherian local ring that admits a surjection from
an $n$-dimensional regular local ring $(R,\mathfrak{m})$ containing its residue field $\K$, and $J \subseteq R$ be the kernel
of the surjection. Then, the Bass numbers 
$$\lambda_{p,i} (A):= \mu_p(\fM, H_{J}^{n-i}(R))$$ 
depend only on $A$, $i$ and $p$, but not on the choice
of $R$ or the surjection $R\lra A$. More generally, all the Bass numbers $\mu_p(\fp, H_{J}^{n-i}(R))$ are invariants of the local ring $A$ as it was proved later on in \cite{Al04}. Bass numbers behave well with respect to completion so we may always assume that $A$ is a quotient of a formal power series ring $R$.  Considering a squarefree monomial ideal as an ideal in the polynomial or the formal power series ring makes no difference since 
the Bass numbers of the corresponding  local cohomology modules coincide.  Is for this reason that we will keep considering, for simplicity, just the case of $R$ being a polynomial ring.

\vskip 2mm

 Lyubeznik numbers satisfy $\la_{d,d}(A)\neq 0$ and $\la_{p,i}(A) \ne 0$ implies  $0 \le p \le i \le d$, where $d=\dm  A$.  
%
%
%
A way to collect these invariants is by means of 
the so-called {\bf Lyubeznik table}:
$$\Lambda(A)  = \left(
                    \begin{array}{ccc}
                      \la_{0,0} & \cdots & \la_{0,d}  \\
                       & \ddots & \vdots \\
                       &  & \la_{d,d} \\
                    \end{array}
                  \right)
$$ and we say that the Lyubeznik table is {\bf trivial} if $ \la_{d,d}=1$ and the rest of these invariants vanish.

The highest Lyubeznik number $\la_{d,d} (A)$ has an interesting interpretation in terms of the so-called {\bf Hochster-Huneke graph}
defined in \cite{HH},
which is the graph whose vertices are the minimal primes of $A$ and we have an edge between two vertices $\fp$ and $\mathfrak{q}$ if and only if ${\rm ht}(\fp + \mathfrak{q})=1$. The following result was proved by G.~Lyubeznik \cite{Ly06} when $\KK$  is a 
positive characteristic field,  and a characteristic- free proof was given by W.~Zhang \cite{Zha07}. To avoid technicalities in the statement of the result we will restrict ourselves to the case of squarefree monomial ideals in a polynomial ring.

\begin{theorem}
Let $I \subseteq R=\K[x_1, \ldots, x_n]$ be a squarefree monomial ideal and $A=R/I$.
The highest Lyubeznik number $\lambda_{d,d}(A)$ equals the connected components of the Hochster-Huneke graph of $A$.
\end{theorem}

\vskip 2mm

Another property that we are  going to use in this work is the following 
Thom-Sebastiani type formula for the case of squarefree monomial ideals that was proved in \cite{AY18}.

\begin{proposition}\label{dual_join}
Let $I \subseteq R=\K[x_1, \ldots, x_m]$ and  $J \subseteq
S=\K[y_1, \ldots, y_n]$ be squarefree monomial ideals in two
disjoint sets of variables and set $T= \K[x_1, \ldots, x_m, y_1, \ldots, y_n]$.
Then, the Lyubeznik numbers of $T/IT\cap JT$  have the following form:

\begin{itemize}
\item[i)]   If either the height of $I$ or the height of $J$ is  $1$, then $T/IT \cap JT$ has trivial Lyubeznik table.

\item[ii)]  If both the height of $I$ and the height of $J$ are  $\geq 2$, then we have:
\begin{eqnarray*}
\lambda_{p,i}(T/IT \cap JT) &=& \lambda_{p,i}(T/IT) + \lambda_{p,i}(T/JT) +
\sum_{\substack{q+r=p+\dim T\\j+k=i+\dim T-1}} \lambda_{q,j}(T/IT) \lambda_{r,k}(T/JT)\\
&=& \lambda_{p-n,i-n}(R/I) + \lambda_{p-m,i-m}(S/J) +
\sum_{\substack{q+r=p\\j+k=i-1}} \lambda_{q,j}(R/I) \lambda_{r,k}(S/J).
\end{eqnarray*}
\end{itemize}
\end{proposition}

The following particular case will be very useful later on.

\begin{corollary}\label{disjoint_trivial}
Let $I \subseteq R=\K[x_1, \ldots, x_n]$ be a squarefree monomial ideal admitting a decomposition $I=I_1 \cap \cdots \cap I_c$
in disjoint sets of variables such that  $\dim R/I_j=d$ and $\Lambda(R/I_j)$ is trivial for  $j=1,\dots , c$. Then 
$$\lambda_{d-2k,d-k}(R/I)= {c \choose k+1} \hskip 5mm {\rm for} \hskip 5mm k=0,\dots , c -1$$
and the rest of Lyubeznik numbers are zero.
\end{corollary}

\begin{proof}
First we notice that the matrices $\Lambda(R/I_i)$ have the same size for all  $i$. 
In the case that $c =2$ we have, using Proposition \ref{dual_join},  $\lambda_{d,d}=2$ and $\lambda_{d-2,d-1}=1$. Then we proceed using  induction on the number of components. 
\end{proof}

A general formula for the case of $c$ disjoint sets of variables could be worked out but we will just focus  on finding the smallest integer $i$ for which there exist $p$ such that $\lambda_{p,i}(R/I)\neq 0$.

\begin{corollary}\label{disjoint_top}
Let $I \subseteq R=\K[x_1, \ldots, x_n]$ be a squarefree monomial ideal admitting a decomposition $I=I_1 \cap \cdots \cap I_c$
in disjoint sets of variables. Let $i_j$ be the smallest integer for which there exist $p$ such that  $\lambda_{p,i_j}(R/I_j)\neq 0$ for  $j=1,\dots , c$. Then,  the smallest $i$ for which there exist $p$ such that $\lambda_{p,i}(R/I)\neq 0$ is  
$i=(i_1+\cdots + i_c)+(c-1). $
\end{corollary}

\begin{proof}
In the case that $c =2$, let $q$ and $r$ be integers such that $\lambda_{q,i_1}(R/I_1)\neq 0$, $\lambda_{r,i_2}(S/I_2)\neq 0$. Then  
$\lambda_{q+r,i_1+i_2+1}(T/I_1T+I_2T)\neq 0$ using Proposition \ref{dual_join} and it gives the smallest integer $i$ satisfying this property. Then we proceed using  induction on the number of components. 
\end{proof}

\subsection{Local cohomology modules and free resolutions} \label{free}
A way to interpret Lyubeznik numbers for the case of squarefree monomial ideals  is in terms of the linear strands of the free resolution of the Alexander dual of the ideal.  This approach was given in \cite{AV14} and further developed in \cite{AY18} and we will briefly recall it here.

\vskip 2mm

Let $J^\vee$ be the Alexander dual of a squarefree monomial ideal $J\subseteq R$. Its 
 minimal  {\bf $\bZ$-graded free resolution} is an
exact sequence of free ${\mathbb{Z}}$-graded $R$-modules:
$$\mathbb{L}_{\bullet}(J^\vee): \hskip 3mm \xymatrix{ 0 \ar[r]& L_{m}
\ar[r]^{d_{m}}& \cdots \ar[r]& L_1 \ar[r]^{d_1}& L_{0} \ar[r]& J^\vee
\ar[r]& 0}$$  where the $j$-th term is of the form $$L_j =
\bigoplus_{\ell \in {\mathbb{Z}}}
R(-\ell)^{\beta_{j,\ell}(J^\vee)},$$ and the matrices of the
morphisms $d_j: L_j\longrightarrow L_{j-1}$ do not contain
invertible elements.  The $\bZ$-graded {\bf Betti numbers} of $J^\vee$
are the invariants $\beta_{j,\ell}(J^\vee)$. Given an integer $r$, the
{\bf $r$-linear strand} of $\mathbb{L}_{\bullet}(J^\vee)$ is the complex:
$$\mathbb{L}_{\bullet}^{<r>}(J^\vee): \hskip 3mm \xymatrix{ 0 \ar[r]&
L_{n-r}^{<r>} \ar[r]^{d_{n-r}^{<r>}}& \cdots \ar[r]& L_1^{<r>}
\ar[r]^{d_1^{<r>}}& L_{0}^{<r>} \ar[r]& 0},$$ where $$L_j^{<r>} =  R(-j -r)^{\beta_{j,j+r}(J^\vee)} ,$$ and the
differentials $d_j^{<r>}: L_j^{<r>}\longrightarrow L_{j-1}^{<r>}$
are the corresponding components of $d_j$. 

\vskip 2mm

We point out that these differentials can be described using the so-called {\bf monomial
matrices} introduced by E.~Miller in \cite{Mi00} (see also \cite{MS05}).  These are matrices with scalar entries that keep track
of the degrees of the generators of the summands in the source and
the target.  Now  we construct a complex of $\K$-vector spaces
$$\mathbb{F}_{\bullet}^{<r>}(J^{\vee})^{\ast}: \hskip 3mm \xymatrix{ 0 &
{\underbrace{\K^{\beta_{n-r,n}(J^\vee) }}_{\deg 0}}\ar[l]& \cdots \ar[l]& {\underbrace{\K^{\beta_{1,1+r}(J^\vee)}}_{\deg n-r-1} }\ar[l]&
{\underbrace{\K^{\beta_{0,r}(J^\vee)} }_{\deg n-r}} \ar[l]& 0 \ar[l]}$$
where the morphisms are given by the transpose of the corresponding monomial matrices and thus we reverse the indices of the complex.
Then, the Lyubeznik numbers are described by means of the homology groups of these complexes. Namely, the result given in \cite[Corollary 4.2]{AV14} is the following characterization
 \begin{equation}
 \lambda_{p,n-r}(R/J)= {\rm
dim}_{\K} H_{p}(\mathbb{F}_{\bullet}^{<r>}(J^{\vee})^{\ast}).
\end{equation}

\subsection{Mayer-Vietoris splitting} \label{split1}
A successful technique used in the study of free resolutions of monomial ideals 
was developed by S.~Eliahou and M.~Kervaire in \cite{EK}  and refined  by C.~Francisco, H.~T.~H\`a and  A.~Van Tuyl  in \cite{FHV} under the terminology of {\bf splittings} of monomial ideals and {\bf Betti splittings} respectively.

\vskip 2mm

An analogous technique can be used to study local cohomology modules.

\begin{definition}\label{def_MV}
Let $J \subseteq R$ be a squarefree monomial ideal. We say that the decomposition $J=L\cap K$ is a MV-splitting if the Mayer-Vietoris sequence 
$$\cdots \lra H^r_{L+K}(R) \lra H^r_{L}(R)\oplus  H^r_{K}(R) \lra  H^r_{J}(R) \lra  H^{r+1}_{L+K}(R) \lra \cdots$$
splits into short exact sequences
$$0\lra H^r_{L}(R)\oplus H^r_{K}(R) \lra  H^r_{J}(R) \lra  H^{r+1}_{L+K}(R) \lra 0$$ for all $r$.
\end{definition}

\begin{remark}
Using Alexander duality, we have that this notion is equivalent to the fact that $J^\vee$ admits a Betti splitting $J^{\vee}=L^{\vee} + K^{\vee}$ in the sense of \cite{FHV}, 
which means that the $\bZ^n$-graded Betti numbers satisfy
$$\beta_{i,\alpha}(J^{\vee})= \beta_{i,\alpha}(L^{\vee})+\beta_{i,\alpha}(K^{\vee})+\beta_{i-1,\alpha}(L^{\vee}\cap
K^{\vee})$$ 
\end{remark}

Certainly we have a MV-splitting if the $\bZ^n$-graded morphisms $ H^r_{L+K}(R) \lra H^r_{L}(R)\oplus  H^r_{K}(R)$ are zero for all $r$.
Sufficient conditions for this vanishing can be given in terms of the posets of sums of ideals associated to $L$, $K$ and $L+K$. 
 The following result, which uses the terminology of Equation (\ref{pieces}), can be understood as a reinterpretation of  \cite[Theorem 2.3]{FHV}.
 
 \begin{proposition}\label{posets1}
 Let $J=L\cap K$ be a decomposition of a squarefree monomial ideal $J \subseteq R$. Consider the posets $\cP_L, \cP_K$ and $\cP_{L+K}$
 associated to the primary decompositions of the ideals $L,K$ and $L+K$ respectively. Assume that  $m_{r,\alpha}(L+K)\neq 0$ implies
 $m_{r,\alpha}(L)= m_{r,\alpha}(K)=0$ for any $r$ and any $\alpha \in \{0,1\}^n$. Then the decomposition $J=L\cap K$ is a MV-splitting.
 \end{proposition}
 
 \begin{proof} 
 The assumptions we are considering are telling us that  $[H_{L+K}^r(R)]_{-\alpha}\neq 0$ implies $[H_{L}^r(R)]_{-\alpha}= [H_{K}^r(R)]_{-\alpha}=0$  by means of Equation (\ref{pieces}),  and thus the $\bZ^n$-graded morphisms $ H^r_{L+K}(R) \lra H^r_{L}(R)\oplus  H^r_{K}(R)$ are zero for all $r$.
 \end{proof}
 
 \begin{corollary} \label{posets2}
 Let $J=L\cap K$ be a decomposition of a squarefree monomial ideal $J \subseteq R$. Assume that the posets $\cP_L, \cP_K$ and $\cP_{L+K}$
 associated to the primary decompositions of the ideals $L,K$ and $L+K$  have no face ideal in common. Then the decomposition $J=L\cap K$ is a MV-splitting.
 \end{corollary}

We want to apply these splitting techniques to the study of the composition of local cohomology modules. The following discussion will be crucial in the rest of this work.

\begin{discussion}\label{MVsplit}
 The degree {\bf - 1} part of the  long exact sequence of local cohomology associated to the short exact sequences
 \begin{equation}\label{MV}
  0\lra H^r_{L}(R)\oplus H^r_{K}(R) \lra  H^r_{J}(R) \lra  H^{r+1}_{L+K}(R) \lra 0
 \end{equation}

\noindent  obtained in a MV-splitting is 
\begin{equation}\label{long}
{\small 
\begin{aligned}
 \cdots \longrightarrow & [H^{p-1}_{\m}(H_{L+K}^{r+1}(R))]_{\bf -1}  \xrightarrow{\partial^r_{p-1}}   [H^{p}_{\m}(H_{L}^{r}(R))]_{\bf -1} \oplus [H^{p}_{\m}(H_{K}^{r}(R))]_{\bf -1}  \longrightarrow  [H^{p}_{\m}(H_{J}^{r}(R))]_{\bf -1} \longrightarrow \\
 & \longrightarrow [H^{p}_{\m}(H_{L+K}^{r+1}(R))]_{\bf -1}  \xrightarrow{\partial^r_p}  
   [H^{p+1}_{\m}(H_{L}^r(R))]_{\bf -1} \oplus [H^{p+1}_{\m}(H_{K}^{r}(R))]_{\bf -1}  \longrightarrow \cdots 
\end{aligned}}
\end{equation}

Equivalently, it is the long exact sequence of $\K$-vector spaces whose dimensions are the corresponding Lyubeznik numbers. Namely,
\begin{equation}\label{long_k}
{\small \begin{aligned}
 \cdots \longrightarrow & \K^{\lambda_{p-1,n-r-1}(R/L+K)}  \xrightarrow{\partial^r_{p-1}}    \K^{\lambda_{p,n-r}(R/L)} \oplus  \K^{\lambda_{p-1,n-r}(R/K)}  \longrightarrow   \K^{\lambda_{p-1,n-r}(R/J)}\longrightarrow \\
 & \longrightarrow  \K^{\lambda_{p,n-r-1}(R/L+K)} \xrightarrow{\partial^r_p}  
    \K^{\lambda_{p+1,n-r}(R/L)} \oplus  \K^{\lambda_{p+1,n-r}(R/K)}  \longrightarrow \cdots
\end{aligned}}
\end{equation}
Therefore, if we want to compute the Lyubeznik numbers of $R/J$ in terms of the Lyubeznik numbers of  $R/L$,  $R/K$ and  $R/L+K$, we need to control the connecting morphisms $\partial^r_p$'s. 

\end{discussion}
Using the methods considered in \cite{AV14} we may give an interpretation of these differentials in terms of linear strands. First, the
short exact sequence (\ref{MV}) corresponds to the short exact sequence of complexes of $\K$-vector spaces
$$0 \longleftarrow  \mathbb{F}_{\bullet}^{<r>}(L^{\vee})^{\ast} \oplus \mathbb{F}_{\bullet}^{<r>}(K^{\vee})^{\ast} \longleftarrow 
\mathbb{F}_{\bullet}^{<r>}(J^{\vee})^{\ast} \longleftarrow \mathbb{F}_{\bullet}^{<r+1>}((L+K)^{\vee})^{\ast}\longleftarrow 0$$
and the long exact sequence (\ref{long}) corresponds to 

{\small \begin{align*}
 \cdots \longleftarrow & H_{p-1}(\mathbb{F}_{\bullet}^{<r+1>}((L+K)^{\vee})^{\ast})  \xleftarrow{\partial^r_{p-1}}   H_{p}(\mathbb{F}_{\bullet}^{<r>}(L^{\vee})^{\ast}) \oplus H_{p}(\mathbb{F}_{\bullet}^{<r>}(K^{\vee})^{\ast})  \longleftarrow  H_{p}(\mathbb{F}_{\bullet}^{<r>}(J^{\vee})^{\ast}) \longleftarrow \\
 & \longleftarrow H_{p}(\mathbb{F}_{\bullet}^{<r+1>}((L+K)^{\vee})^{\ast})   \xleftarrow{\partial^r_p}  
  H_{p+1}(\mathbb{F}_{\bullet}^{<r>}(L^{\vee})^{\ast}) \oplus H_{p+1}(\mathbb{F}_{\bullet}^{<r>}(K^{\vee})^{\ast})  \longleftarrow \cdots
\end{align*}}

\section{Lyubeznik tables of cover ideals of graphs} \label{Lyubeznik}
Let $G=(V_G,E_G)$ be a simple finite graph in the set of vertices $V_G=\{x_1,\dots , x_n\}$ and the set of edges $E_G$. For simplicity we will also assume that $G$ is connected. For a vertex $x_i$, we consider its {\bf neighbour set} $N(x_i)=\{ x_j \in G \hskip 2mm | \hskip 2mm \{x_i,x_j\} \in E_G\}$. The {\bf degree} of a vertex is the cardinal of its neighbour set.

Let  $J(G) \subseteq R$ be the cover ideal of $G$ where $R=\K[x_1,\dots,x_n]$ is a polynomial ring with coefficients in a field $\K$. In this section we will develop MV-splitting techniques to study the Lyubeznik numbers of  $R/J(G)$. 
To start with, we recall that since $J(G)$ is a pure height two ideal, all the entries in the main diagonal of the Lyubeznik table are zero except for the highest Lyubeznik number (see \cite{Al00} for details).  

\begin{lemma}\label{highest}
Let $J(G)$ be the cover ideal of a simple connected graph $G$. Then, the highest Lyubeznik number is $\lambda_{d,d}(R/J(G))=1$
\end{lemma}

\begin{proof}
The vertices of the Hochster-Huneke graph  of $J(G)$ correspond to the edges of $G$, and the edges of the Hochster-Huneke graph correspond to adjacent edges of $G$.
Therefore, the Hochster-Huneke graph has just one connected component since the graph $G$ is connected.
\end{proof}
Under these restrictions, the shape of the Lyubeznik table is 
$$\Lambda(R/J(G))  = \left(
                    \begin{array}{ccccc}
                     0&  \la_{0,1} & \cdots & \la_{0,d-1}& \la_{0,d}  \\
                      &  0 & \cdots & \la_{1,d-1}& \la_{1,d}  \\
                      &  & \ddots & \vdots & \vdots\\
                     &   &  & 0 & \la_{d-1,d} \\
                     &   &  &  & 1 \\
                    \end{array}
                  \right)
$$

In the case that $R/J(G)$ is Cohen-Macaulay we have that the Lyubeznik table is trivial (see \cite[Remark 4.2]{Al00}). Recall that, combining the results in \cite{ER98} with \cite{Fro}, we have the following characterization of this property.

\begin{proposition}
Let $G$ be a simple graph. Then the following are equivalent:

\begin{itemize}
 \item[i)] The cover ideal $J(G)$ is Cohen-Macaulay.
  \item[ii)] The edge ideal $I(G)$ has a linear resolution. 
   \item[iii)] The complement graph $G^c$ is chordal.

\end{itemize}

\end{proposition}

Free resolutions of edge ideals have been extensively studied over the last years and we may find in the literature 
several families of Cohen-Macaulay graphs. For example, 

\begin{itemize}
 \item [$\cdot$] {\bf Complete graphs} $K_n$ \cite{Jac04}.
 
 \item [$\cdot$] {\bf Complete bipartite graphs} $K_{n,m}$ and in particular {\bf star graphs} $K_{1,m}$ \cite{Jac04}.
 
 \item [$\cdot$] {\bf Ferrers graphs} \cite{CN}.
\end{itemize}

The simplest examples of ideals with non-trivial Lyubeznik table are
minimal non-Cohen-Macaulay squarefree monomial ideals (see
\cite{Ly}).  The unique minimal non-Cohen-Macaulay squarefree monomial ideal of
pure height two in $R=\K[x_1,\dots,x_n]$ is the cover ideal of the {\bf complement of a cycle}:
$$J({C}^c_n)= (x_1, x_3)\cap\cdots\cap(x_1, x_{n- 1})\cap(x_2, x_4)\cap\cdots\cap(x_2, x_n)\cap(x_3, x_5)\cap\cdots\cap(x_{n- 2}, x_n).$$ Its Lyubeznik table is of the form (see \cite{AV14})
{ $$\Lambda(R/J({C^c_n}))  = \begin{pmatrix}
  0 & 0 & 0 & \cdots & 0 & 1 & 0 \\
    & 0 & 0 & \cdots & 0 & 0 & 0 \\
    &   & 0 & & 0 & 0 & 1 \\
    &   & & \ddots &  & 0 & 0 \\
&   &   & &  & \vdots & \vdots \\
    &   & &  &   & 0 & 0 \\
    &   & &  &   &   & 1
\end{pmatrix}$$}

To provide a full description of all the possible Lyubeznik tables of cover ideals of graphs is completely out of the scope of this work. 
Our aim is to introduce some Mayer-Vietoris splitting techniques that will allow us to compute large families of examples. To such purpose we will follow
the ideas considered in Discussion \ref{MVsplit}. To start with, we consider the case where $J(G)=L\cap K$ is a MV-splitting with the extra assumption that
the Lyubeznik table of $R/K$ is trivial.

\begin{theorem} \label{main}
Let $J(G) \subseteq R$ be the cover ideal of a simple connected graph $G$.
Let $J(G)=L\cap K$ be a MV-splitting such that $\Lambda(R/K)$ is trivial. Then:

\begin{itemize}
\item[i)] If $\Lambda(R/L)$ and $\Lambda(R/L+K)$ are trivial, then $\Lambda(R/J(G))$ is trivial.
\item[ii)] If $\Lambda(R/L+K)$ is trivial, then  $\Lambda(R/J(G))=\Lambda(R/J)$.
\item[iii)] If $\Lambda(R/L)$ is trivial and $$\Lambda(R/L+K)  = \left(
                    \begin{array}{ccc}
                      \cellcolor{usc!20} \la'_{0,0} & \cellcolor{usc!20}\cdots & \cellcolor{usc!20}\la'_{0,d-1}  \\
                       & \cellcolor{usc!20}\ddots & \cellcolor{usc!20}\vdots \\
                       &  & \cellcolor{usc!60}\la'_{d-1,d-1} \\
                    \end{array}
                  \right),
$$ then the Lyubeznik table of $R/J(G)$ is 
$$\Lambda(R/J(G))  = \left(
                    \begin{array}{ccccc}
                     0&  \cellcolor{usc!20}\la'_{0,0} & \cellcolor{usc!20}\cdots & \cellcolor{usc!20} \la'_{0,d-2}& \cellcolor{usc!20}\la'_{0,d-1}  \\
                      &  0 & \cellcolor{usc!20}\cdots & \cellcolor{usc!20}\la'_{1,d-2}& \cellcolor{usc!20}\la'_{1,d-1}  \\
                      &  & \ddots & \cellcolor{usc!20}\vdots & \cellcolor{usc!20}\vdots\\
                     &   &  & 0 &  \cellcolor{usc!60} \la'_{d-1,d-1}-1 \\
                     &   &  &  & 1 \\
                    \end{array}
                  \right).
$$ 

\end{itemize}

\end{theorem}

\begin{proof}
Assume that $\Lambda(R/K)$ is trivial and recall that, using  Lemma \ref{highest}, the highest Lyubeznik number of the cover ideal of a graph is one. Then, for $r=2$, the long exact sequence \ref{long} considered in 
Discussion \ref{MVsplit} 
{\small \begin{align*}
 \cdots \longrightarrow & [H^{n-3}_{\m}(H_{L+K}^{3}(R))]_{\bf -1}  \xrightarrow{\partial^2_{n-3}}   [H^{n-2}_{\m}(H_{L}^{2}(R))]_{\bf -1} \oplus [H^{n-2}_{\m}(H_{K}^{2}(R))]_{\bf -1}  \longrightarrow  [H^{n-2}_{\m}(H_{J(G)}^{2}(R))]_{\bf -1} \longrightarrow 0
\end{align*}}
turns out to be
{\small \begin{align*}
 \cdots \longrightarrow &  \K^{\lambda_{n-3,n-2}(R/L)} \longrightarrow \K^{\lambda_{n-3,n-2}(R/J(G))} \longrightarrow \K^{\lambda_{n-3,n-3}(R/L+K)}  \xrightarrow{\partial^2_{n-3}}   \K \oplus \K  \longrightarrow  \K \longrightarrow 0
\end{align*}}
Moreover, for  $r>2$ and any $p$, the long exact sequence becomes 
{\small \begin{align*}
 \cdots \longrightarrow & \K^{\lambda_{p-1,n-(r+1)}(R/L+K)}  \xrightarrow{\partial^r_{p-1}} \K^{\lambda_{p,n-r}(R/L)} \longrightarrow \K^{\lambda_{p,n-r}(R/J(G))} \longrightarrow \K^{\lambda_{p,n-(r+1)}(R/L+K)}  \xrightarrow{\partial^r_{p}}   \cdots
\end{align*}}

Now we are ready to consider all the cases:
\begin{itemize}
 \item[i)] If $ \Lambda(R/L)$ and $ \Lambda(R/L+K) $ are trivial then the Lyubeznik table of $ R/J(G) $ is trivial as well. Notice that for $r=2$ we have
 {\small \begin{align*}
 \cdots  \longrightarrow \K^{\lambda_{n-3,n-2}(R/J(G))} \longrightarrow \K \xrightarrow{\partial^2_{n-3}}   \K \oplus \K  \longrightarrow  \K \longrightarrow 0
\end{align*}}
and thus ${\lambda_{n-3,n-2}(R/J(G))}=0$ and the vanishing of the rest of Lyubeznik numbers follow immediately.
 
 \vskip 2mm 
 
\item[ii)] If $\Lambda(R/L+K)$  is trivial then we have 
{\small \begin{align*}
 0 \longrightarrow &  \K^{\lambda_{n-3,n-2}(R/L)} \longrightarrow \K^{\lambda_{n-3,n-2}(R/J(G))} \longrightarrow \K \xrightarrow{\partial^2_{n-3}}   \K \oplus \K  \longrightarrow  \K \longrightarrow 0
\end{align*}}
and thus ${\lambda_{n-3,n-2}(R/L)}= {\lambda_{n-3,n-2}(R/J(G))}$. The rest of Lyubeznik numbers also coincide so we get $\Lambda(R/J(G))=\Lambda(R/J)$.

 \vskip 2mm 
 
\item[iii)] If $\Lambda(R/L)$  is trivial then we have
{\small \begin{align*}
 0 \longrightarrow \K^{\lambda_{n-3,n-2}(R/J(G))} \longrightarrow \K^{\lambda_{n-3,n-3}(R/L+K)}  \xrightarrow{\partial ^2_{n-3}}   \K \oplus \K  \longrightarrow  \K \longrightarrow 0
\end{align*}}
and thus ${\lambda_{n-3,n-2}(R/J(G))} = \K^{\lambda_{n-3,n-3}(R/L+K)} - 1$. The rest of Lyubeznik numbers satisfy ${\lambda_{p,n-r}(R/J(G))}=\lambda_{p,n-(r+1)}(R/L+K)$ and the result follows. 
\end{itemize}
\end{proof}

\subsection{Splitting vertices} \label{Lyubeznik2}

Let $J(G) \subseteq R$ be the cover ideal of a simple connected graph $G$. 
The easiest way to provide a MV-splitting $J(G)=L\cap K$ satisfying that the Lyubeznik table of $R/K$ is trivial is by means of a {\bf splitting vertex}. Namely, we fix a vertex, say $x_n$, and we decompose the ideal $J(G)$ depending on the edges that contain this vertex.
$$J(G)= \underbrace{\left(\underset{{x_{i},x_{j} \not \in N_{G}(x_{n})}}{\bigcap}(x_{i}, x_{j}) \right)}_{L}  \cap  \underbrace{\left(\underset{x_{k} \in N_{G}(x_{n})}{\bigcap}(x_{k}, x_{n}) \right)}_{K} $$
Notice that we have:

\begin{itemize}
\item[$\cdot$] $L = J(G\setminus \{x_n\})$ is the cover ideal of the subgraph obtained removing the vertex $x_n$.
\item[$\cdot$] $K = J(K_{1,g})$ is the cover ideal of a star graph with $g=\deg(x_n)$. 
\item[$\cdot$] $L+K$ is a height $3$ monomial ideal which admits a (non-necessarily minimal) primary decomposition of the form:
\end{itemize}

$$ L+K= \underset{x_{k} \in N_{G}(x_{n})}{\bigcap}\left[ \left(\underset{x_{i}, x_{j} \not \in N_{G}(x_{n}) , N_{G}(x_{k})}{\bigcap}(x_{i}, x_{j},x_{k}, x_{n})\right) \cap  \left(\underset{x_{l} \in N_{G}(x_{k})}{\bigcap}(x_{l},x_{k}, x_{n})\right)\right].$$\\

Of course we can make it minimal removing conveniently the extra components. Notice that $\Lambda(R/K)$ is trivial.
In order to check that this decomposition indeed  provides a MV-splitting we only need to invoke \cite[Theorem 4.2]{HV} where it is proved that 
every vertex is a splitting vertex except for some limit cases where the vertex is isolated or its complement consists of isolated vertices.


\subsubsection{Splitting vertices of degree one}
Let $x_n$ be a splitting vertex of a graph $G$. Assume that its degree is one and, for simplicity, we will take $x_{n-1}$ as the unique vertex in its neighbourhood. We can rephrase it by saying that we are adding a {\bf whisker} to the vertex $x_{n-1}$ of the graph $G\setminus \{x_n\}$.

\begin{proposition} \label{whisker}
Let $J(G) \subseteq R$ be the cover ideal of a simple connected graph $G$. Let $x_n\in V_G$ be a vertex of degree one.
Then $\Lambda(R/J(G))=\Lambda(R/J({G\setminus \{x_n\}}))$.
 
\end{proposition}

\begin{proof}
Let $ x_{n-1} $ be the unique vertex in the neighbourhood of $ x_{n}. $ Then, the MV-splitting $J_{G}= L\cap K $  given by $x_n$ has $K=( x_{n-1},x_{n}) $ and
\begin{align*}
 L+K&= \left(\underset{x_{i}, x_{j} \not \in N_{G}(x_{n-1})}{\bigcap}(x_{i}, x_{j},x_{n-1}, x_{n})\right) \cap  \left(\underset{x_{l} \in N_{G}(x_{n-1})}{\bigcap}(x_{l},x_{n-1}, x_{n})\right)\\
&=\underbrace{\left[ \left(\underset{x_{i}, x_{j} \not \in N_{G}(x_{n-1})}{\bigcap}(x_{i}, x_{j})\right) \cap  \left(\underset{x_{l} \in N_{G}(x_{n-1})}{\bigcap}(x_{l})\right)\right]}_{M} +(x_{n-1}, x_{n})
\end{align*}
The ideal $M$ is a height one ideal in two sets of disjoint variables. Therefore, its Lyubeznik table is trivial because of Proposition \ref{dual_join}. Given the isomorphism  
$$\K[x_1,\dots, x_n]/L+K \cong \K[x_1,\dots, x_{n-2}]/M$$ we get that $\Lambda(R/L+K)$ is trivial as well. Then the result follows using Theorem \ref{main} and the fact that $L=J({G\setminus \{x_n\}})$.
\end{proof}

\subsubsection{Splitting vertices of degree two} 
Let $x_n$ be a splitting vertex of a graph $G$. Assume that its degree is two and the vertices in its neighbourhood are $x_{n-1}$ and $x_{n-2}$.
In this case  we also have the invariance of the Lyubeznik table after removing the splitting vertex under certain extra conditions.

\begin{proposition}\label{handle}
Let $J(G) \subseteq R$ be the cover ideal of a simple connected graph $G$. Let $x_n\in V_G$ be a vertex of degree $2$ with $N_G(x_n)=\{x_{n-2},x_{n-1}\}$. If any of the following  conditions hold:
\begin{itemize}
\item[i)] $\{x_{n-1},x_{n-2}\}\in E_G$,
\item[ii)] there exists $x_c\in N_G(x_{n-1}) \cap N_G(x_{n-2})$,
\end{itemize}
then  $\Lambda(R/J(G))=\Lambda(R/J({G\setminus \{x_n\}}))$.
 
\end{proposition}

\begin{proof}
We have a MV-splitting $J(G)= L\cap K $ where $ K=(x_{n-{1}}, x_{n}) \cap ( x_{n-{2}} , x_{n}) $ and
 {\small \begin{align*}
  L&=\left(\underset{x_{i}, x_{j} \not \in N_{G}(x_{n})}{\bigcap}(x_{i}, x_{j})\right) {\cap}  \left(
 \underset{\substack{x_{a} \in N_{G}(x_{n-{1}}) \\ x_{a} \not \in N_{G}(x_{n-{2}})}}{\bigcap} 
   (x_{a},x_{n-{1}})\right) 
   {\cap} 
\left(
 \underset{\substack{x_{b} \in N_{G}(x_{n-{2}}) \\ x_{b} \not \in N_{G}(x_{n-{1}})}}{\bigcap} 
   (x_{b},x_{n-{2}})\right)  \\
 &  {\cap} 
\left(
 \underset{\stackrel{x_{c} \in N_{G}(x_{n-{1}})}{ x_{c} \in N_{G}(x_{n-{2}})}}{\bigcap} 
  (x_{c},x_{n-{2}})\cap (x_{c},x_{n-{1}})\right) \cap (x_{n-2},x_{n-1})
 \end{align*}}
Under the assumptions we are considering, at least one of the last components in this decomposition must appear. 
Therefore
 {\small \begin{align*}
  L+K&=\underbrace{\left(\underset{\substack{x_{i}, x_{j} \not \in N_{G}(x_{n}) \\x_{i}, x_{j} \not \in N_{G}(x_{n-1})}}{\bigcap}(x_{i}, x_{j},x_{n-{1}}, x_{n})\right) {\cap} 
\left(
 \underset{\substack{x_{a} \in N_{G}(x_{n-{1}}) \\ x_{a} \not \in N_{G}(x_{n-{2}})}}{\bigcap} 
   (x_{a},x_{n-{1}},x_n)\right) }_{M}  \\
& \underbrace{{\cap} 
\left(
 \underset{\substack{x_{c} \in N_{G}(x_{n-{1}}) \\ x_{c} \in N_{G}(x_{n-{2}})}}{\bigcap} 
   (x_{c},x_{n-{1}},x_n)\right) \cap (x_{n-2},x_{n-1})}_{M}\\
 &{\cap}\underbrace{\left(\underset{\substack{x_{i}, x_{j} \not \in N_{G}(x_{n}) \\ x_{i}, x_{j} \not \in N_{G}(x_{n-1})}}{\bigcap}(x_{i}, x_{j},x_{n-{2}}, x_{n})\right) {\cap} 
\left(
 \underset{\substack{x_{b} \in N_{G}(x_{n-{2}}) \\ x_{b} \not \in N_{G}(x_{n-{1}})}}{\bigcap} 
   (x_{b},x_{n-{2}},x_n)\right) }_{N}
 \end{align*}}


Notice that we can rephrase the ideals $M$ and $N$ as

$M= \left[\left(\underset{\substack{x_{i}, x_{j} \not \in N_{G}(x_{n}) \\ x_{i}, x_{j} \not \in N_{G}(x_{n-1})}}{\bigcap}(x_{i}, x_{j})\right) {\cap} 
 \left(\underset{\substack{x_{a} \in N_{G}(x_{n-{1}}) \\ x_{a} \not \in N_{G}(x_{n-{2}})}}{\bigcap} 
   (x_{a})\right) \cap \left(
 \underset{\substack{x_{c} \in N_{G}(x_{n-{1}}) \\ x_{c} \in N_{G}(x_{n-{2}})}}{\bigcap} 
   (x_{c})\right) \cap (x_{n-2}) \right] +(x_{n-{1}}, x_{n})$

$N= \left[\left(\underset{\substack{x_{i}, x_{j} \not \in N_{G}(x_{n}) \\ x_{i}, x_{j} \not \in N_{G}(x_{n-2})}}{\bigcap}(x_{i}, x_{j})\right) {\cap}  \left(\underset{\substack{x_{b} \in N_{G}(x_{n-2}) \\ x_{b} \not \in N_{G}(x_{n-1})}}
{\bigcap}(x_{b})\right)\right] +(x_{n-{2}}, x_{n})$

\noindent so  both $\Lambda(R/M)$ and $\Lambda(R/N)$ are trivial using Proposition \ref{dual_join}. Moreover, 
in the case that condition $i)$ is satisfied, we have

 {\small  \begin{align*}
  M+N&= \Biggl[ \left(\underset{\substack{x_{i}, x_{j} \not \in N_{G}(x_{n}) \\ x_{i}, x_{j} \not \in N_{G}(x_{n-1})}}{\bigcap}(x_{i}, x_{j})\right)  {\cap} 
 \left(\underset{\substack{x_{b} \in N_{G}(x_{n-{2}}) \\ x_{b} \not \in N_{G}(x_{n-{1}})}}{\bigcap} 
   (x_{b})\right) \cap \left(
 \underset{\substack{x_{c} \in N_{G}(x_{n-{1}}) \\ x_{c} \in N_{G}(x_{n-{2}})}}{\bigcap} 
   (x_{c})\right)   \Biggr]
  + (x_{n-{2}},x_{n-{1}}, x_{n})
 \end{align*}}
 
 \noindent Under  condition $ii)$ we have
  {\small  \begin{align*}
  M+N&= \Biggl[ \left(\underset{\substack{x_{i}, x_{j} \not \in N_{G}(x_{n}) \\ x_{i}, x_{j} \not \in N_{G}(x_{n-1})}}{\bigcap}(x_{i}, x_{j})\right)  {\cap} 
 \left(\underset{\substack{x_{a} \in N_{G}(x_{n-{1}}) \\ x_{a} \not \in N_{G}(x_{n-{2}}) \\ x_{b} \in N_{G}(x_{n-{2}}) \\ x_{b} \not \in N_{G}(x_{n-{1}})}}{\bigcap} 
   (x_a,x_{b})\right) \cap \left(
 \underset{\substack{x_{c} \in N_{G}(x_{n-{1}}) \\  x_{c} \in N_{G}(x_{n-{2}})}}{\bigcap} 
   (x_{c})\right)   \Biggr]
  + (x_{n-{2}},x_{n-{1}}, x_{n})
 \end{align*}}
 
In any case the Lyubeznik table of $R/M+N$ is trivial since we are dealing with a height one ideal in a disjoint set of variables so we 
 can apply Proposition \ref{dual_join} once again. As a consequence of Theorem \ref{main}, the Lyubeznik table $\Lambda(R/L+K)$ is trivial and the result follows applying Theorem \ref{main} once more. To finish the proof we need to check that 
the decomposition $ L+K=M\cap N $ is, indeed, a MV-splitting.

 \vskip 2mm
 
If $\{x_{n-2},x_{n-1}\}$ is not an edge of $G$ we have that the variable $ x_{n-{1}} $ appears in all the components of the primary decomposition of $M$ but not in $N$. We also have that $ x_{n-{2}} $ appears in all the components of $N$ but not in $M$ and both
$ x_{n-{2}},  x_{n-{1}} $ appear in $M+N$. In particular the posets associated to $M$, $N$ and $M+N$ do not have common ideals.
Then the result follows from Corollary \ref{posets2}.

 When $\{x_{n-2},x_{n-1}\}$ is an edge of $G$ we have to be more careful. The variable $ x_{n-{1}} $ appears in all the components of  $M+N$ but not in $N$ and thus its corresponding posets do not have common ideals so we only have to compare $M+N$ with $M$. Indeed, since the variables $x_a$'s and $x_b$'s do not belong to both ideals and $x_{n-1}, x_n$ do so we may just assume that the ideals are
 
  {\small  \begin{align*}
  M&= \underbrace{\left(\underset{\substack{x_{i}, x_{j} \not \in N_{G}(x_{n}) \\ x_{i}, x_{j} \not \in N_{G}(x_{n-1})}}{\bigcap}(x_{i}, x_{j})\right) {\cap} 
 \left(
 \underset{\substack{x_{c} \in N_{G}(x_{n-{1}}) \\ x_{c} \in N_{G}(x_{n-{2}})}}{\bigcap} 
   (x_{c})\right) }_{Q}\cap (x_{n-2}) 
   \end{align*}}

 {\small  \begin{align*}
  M+N&= \Biggl[ \left(\underset{\substack{x_{i}, x_{j} \not \in N_{G}(x_{n}) \\ x_{i}, x_{j} \not \in N_{G}(x_{n-1})}}{\bigcap}(x_{i}, x_{j})\right)  {\cap} 
 \left(
 \underset{\substack{x_{c} \in N_{G}(x_{n-{1}}) \\ x_{c} \in N_{G}(x_{n-{2}})}}{\bigcap} 
   (x_{c})\right)   \Biggr]
  + (x_{n-{2}})
 \end{align*}}
 
 \noindent and thus we have a MV-splitting $$0\lra H^r_{Q}(R)\oplus H^r_{(x_{n-2})}(R) \lra  H^r_{M}(R) \lra  H^{r+1}_{M+N}(R) \lra 0$$
 so the graded pieces of $H^{r+1}_{M+N}(R)$ are related to the graded pieces of  $H^r_{M}(R)$ instead of those of $H^{r+1}_{M}(R)$.
 \end{proof}

We can still say something about the Lyubeznik numbers of $J(G)$ in the event that the hypothesis of the previous proposition do not hold.
In this case,  the Lyubeznik table of $R/L+K$ is not trivial so  we need to control the connecting morphisms 
$$\K^{\lambda_{p-1,n-(r+1)}(R/L+K)}  \xrightarrow{\partial^r_{p-1}} \K^{\lambda_{p,n-r}(R/L)}$$ considered in 
Discussion \ref{MVsplit}. If the Lyubeznik table of $L=J(G\setminus \{x_n\})$ is trivial or, more generally,  the connecting morphisms are zero, 
we can give a formula for the Lyubeznik numbers of $J(G)$ in terms of those of $J(G\setminus \{x_n\})$ and the Lyubeznik numbers associated to the graph
$$H:= \left( G\setminus \{x_n\} \right) \cup \bigcup_{\substack{ x_a \in N_{G}(x_{n-{1}})\\ x_b \in  N_{G}(x_{n-{2}})}}\{x_a,x_b\}$$  
whose cover ideal is 
$$J(H)= \left(\underset{\substack{x_{i}, x_{j} \not \in N_{G}(x_{n-2}) \\ x_{i}, x_{j} \not \in N_{G}(x_{n-1})}}{\bigcap}(x_{i}, x_{j})\right)  {\cap} 
 \left(\underset{\substack{x_{a} \in N_{G}(x_{n-{2}})  \\ x_{b} \in N_{G}(x_{n-{1}})  }}{\cap} 
   (x_a,x_{b})\right) $$ 
In other words, the graph $H$ is obtained by adjoining to $G\setminus \{x_n\}$ a complete bipartite graph in the set of vertices  $N_{G}(x_{n-{2}})$ and $N_{G}(x_{n-{1}})$.

\begin{proposition}\label{deg_two}
Let $J(G) \subseteq R$ be the cover ideal of a simple connected graph $G$. Let $x_n\in V_G$ be a vertex of degree $2$ with $N_G(x_n)=\{x_{n-1},x_{n-2}\}$. Assume that conditions $i)$ and $ii)$ of Proposition \ref{handle} no not hold and that the connecting morphisms 
$\partial^r_{p}$ are zero for all $r>2$ and for all $p$. Then we have

\hskip 15mm  $\lambda_{d,d}(R/J(G))=1, \hskip 3mm \lambda_{d-1,d}(R/J(G))= \lambda_{d-1,d}(R/J(G\setminus \{x_n\})) +1,$

\hskip 15mm  $\lambda_{p,r}(R/J(G))= \lambda_{p,r}(R/J(G\setminus \{x_n\})) + \lambda_{p,r-2}(R/J(H))$ \hskip 3mm 

\noindent for $r=2,\dots, d-1$ and $p=0,\dots, r-2$ and the rest of Lyubeznik numbers satisfy 

\hskip 15mm  $\lambda_{p,r}(R/J(G))= \lambda_{p,r}(R/J(G\setminus \{x_n\})) $ .

\vskip 2mm

\noindent That is,

$$\Lambda(R/J(G))  = \left(
                    \begin{array}{ccccccc}
                     \cellcolor{usc!20}\la_{0,0}& \cellcolor{usc!20}\la_{0,1}& \cellcolor{usc!60} \la_{0,2}+\la'_{0,0} &\cellcolor{usc!60} \la_{0,3}+\la'_{0,1} & \cellcolor{usc!60}\cdots & \cellcolor{usc!60}\la_{0,d-1}+\la'_{0,d-3}&\cellcolor{usc!20}\la_{0,d}  \\
                     & \cellcolor{usc!20}\la_{1,1}&  \cellcolor{usc!20}\la_{1,2} &  \cellcolor{usc!60}\la_{1,3}+\la'_{1,1}& \cellcolor{usc!60}\cdots & \cellcolor{usc!60}\la_{1,d-1}+\la'_{1,d-3}& \cellcolor{usc!20}\la_{1,d} \\
                     & &  \cellcolor{usc!20}\la_{2,2} &  \cellcolor{usc!20}\la_{2,3}& \cellcolor{usc!60}\cdots & \cellcolor{usc!60}\la_{2,d-1}+\la'_{2,d-3}& \cellcolor{usc!20}\la_{2,d} \\
                     & &  && \cellcolor{usc!60}\ddots & \cellcolor{usc!60}\vdots & \cellcolor{usc!20}\vdots\\
                      & & &  &  & \cellcolor{usc!60}\la_{d-3,d-1}+\la'_{d-3,d-3} & \cellcolor{usc!20}\la_{d-3,d} \\                
                     & & &  &  & \cellcolor{usc!20}\la_{d-2,d-1} & \cellcolor{usc!20}\la_{d-2,d} \\
                     & & &  &  & \cellcolor{usc!20}\la_{d-1,d-1} & \cellcolor{usc!40}\la_{d-1,d}+1 \\
                     & & & &  &  & \cellcolor{usc!40} 1 \\
                    \end{array}
                  \right)
$$ where $$\Lambda(R/J(G\setminus\{x_n\}))  = \left(
                    \begin{array}{ccc}
                      \cellcolor{usc!20}\la_{0,0} & \cellcolor{usc!20}\cdots & \cellcolor{usc!20}\la_{0,d}  \\
                       & \cellcolor{usc!20}\ddots & \cellcolor{usc!20}\vdots \\
                       &  & \cellcolor{usc!20}\la_{d,d} \\
                    \end{array}
                  \right), \hskip 2mm \Lambda(\K[x_1,\dots, x_{n-3}]/J(H))  = \left(
                    \begin{array}{ccc}
                      \cellcolor{usc!60}\la'_{0,0} & \cellcolor{usc!60}\cdots & \cellcolor{usc!60}\la'_{0,d-3}  \\
                       & \cellcolor{usc!60}\ddots & \cellcolor{usc!60}\vdots \\
                       &  & \cellcolor{usc!60}\la'_{d-3,d-3} \\
                    \end{array}
                  \right)
$$

\end{proposition}

\begin{proof}
Following the same approach as in the proof of Proposition \ref{handle} we have 
 a MV-splitting $J(G)= L\cap K $ with $ K=(x_{n-{1}}, x_{n}) \cap ( x_{n-{2}} , x_{n}) $ but in this case we also have
 {\small \begin{align*}
  L&=\left(\underset{x_{i}, x_{j} \not \in N_{G}(x_{n})}{\bigcap}(x_{i}, x_{j})\right) {\cap}  \left(
 \underset{\substack{x_{a} \in N_{G}(x_{n-{1}}) \\ x_{a} \not \in N_{G}(x_{n-{2}})}}{\bigcap} 
   (x_{a},x_{n-{1}})\right) 
   {\cap} 
\left(
 \underset{\substack{x_{b} \in N_{G}(x_{n-{2}}) \\ x_{b} \not \in N_{G}(x_{n-{1}})}}{\bigcap} 
   (x_{b},x_{n-{2}})\right)  
 \end{align*}}
 
 {\small \begin{align*}
  L+K&=\underbrace{\left(\underset{\substack{x_{i}, x_{j} \not \in N_{G}(x_{n}) \\x_{i}, x_{j} \not \in N_{G}(x_{n-1})}}{\bigcap}(x_{i}, x_{j},x_{n-{1}}, x_{n})\right) {\cap} 
\left(
 \underset{\substack{x_{a} \in N_{G}(x_{n-{1}}) \\ x_{a} \not \in N_{G}(x_{n-{2}})}}{\bigcap} 
   (x_{a},x_{n-{1}},x_n)\right) }_{M}  \\
 &{\cap}\underbrace{\left(\underset{\substack{x_{i}, x_{j} \not \in N_{G}(x_{n}) \\ x_{i}, x_{j} \not \in N_{G}(x_{n-1})}}{\bigcap}(x_{i}, x_{j},x_{n-{2}}, x_{n})\right) {\cap} 
\left(
 \underset{\substack{x_{b} \in N_{G}(x_{n-{2}}) \\ x_{b} \not \in N_{G}(x_{n-{1}})}}{\bigcap} 
   (x_{b},x_{n-{2}},x_n)\right) }_{N}
 \end{align*}}
 
 Once again we rewrite the ideals $M$ and $N$ as

$M= \left[\left(\underset{\substack{x_{i}, x_{j} \not \in N_{G}(x_{n}) \\ x_{i}, x_{j} \not \in N_{G}(x_{n-1})}}{\bigcap}(x_{i}, x_{j})\right) {\cap} 
 \left(\underset{\substack{x_{a} \in N_{G}(x_{n-{1}}) \\ x_{a} \not \in N_{G}(x_{n-{2}})}}{\bigcap} 
   (x_{a})\right) \right] +(x_{n-{1}}, x_{n})$
 
$N= \left[\left(\underset{\substack{x_{i}, x_{j} \not \in N_{G}(x_{n}) \\ x_{i}, x_{j} \not \in N_{G}(x_{n-2})}}{\bigcap}(x_{i}, x_{j})\right) {\cap}  \left(\underset{\substack{x_{b} \in N_{G}(x_{n-2}) \\ x_{b} \not \in N_{G}(x_{n-1})}}
{\bigcap}(x_{b})\right)\right] +(x_{n-{2}}, x_{n})$

\noindent so  both $\Lambda(R/M)$ and $\Lambda(R/N)$ are trivial. Moreover
  {\small  \begin{align*}
  M+N&= \Biggl[ \underbrace{\left(\underset{\substack{x_{i}, x_{j} \not \in N_{G}(x_{n-2}) \\ x_{i}, x_{j} \not \in N_{G}(x_{n-1})}}{\bigcap}(x_{i}, x_{j})\right)  {\cap} 
 \left(\underset{\substack{x_{a} \in N_{G}(x_{n-{1}})  \\ x_{b} \in N_{G}(x_{n-{2}})  }}{\cap} 
   (x_a,x_{b})\right) }_{J(H)}  \Biggr]
  + (x_{n-{2}},x_{n-{1}}, x_{n}).
 \end{align*}}

We have that the variable $ x_{n-{1}} $ appears in all the components of the primary decomposition of $M$ but not in $N$. We also have that $ x_{n-{2}} $ appears in all the components of $N$ but not in $M$ and both
$ x_{n-{2}},  x_{n-{1}} $ appear in $M+N$. In particular the posets associated to $M$, $N$ and $M+N$ do not have common ideals.
Therefore $ L+K=M\cap N $ is a MV-splitting by using Corollary \ref{posets2}.

 \vskip 2mm

\noindent Applying the long exact sequence of local cohomology modules to the short exact sequence  
$$ 0\lra H^r_{M}(R)\oplus H^r_{N}(R) \lra  H^r_{L+K}(R) \lra  H^{r+1}_{M+N}(R) \lra 0 $$
we get, for $r=3$,
{\small \begin{align*}
 \cdots \longrightarrow & H^{n-4}_{\m}(H_{M+N}^{4}(R)) \xrightarrow{\partial^3_{n-4}}   H^{n-3}_{\m}(H_{M}^{3}(R)) \oplus H^{n-3}_{\m}(H_{N}^{3}(R)) \longrightarrow  H^{n-3}_{\m}(H_{L+K}^{3}(R)) \longrightarrow 0
\end{align*}}
Since  $\Lambda(R/M)$ and $\Lambda(R/N)$ are trivial and  taking into account that $\hlt M = \hlt N=3 $ and $\hlt M+N=5$ we get

 $[ H^{n-3}_{\m}(H_{L+K}^{3}(R))]_{_{\bf -1}} \cong \K^{2}$

$ [ H^{p}_{\m}(H_{L+K}^{3}(R))]_{_{\bf -1}} \cong 0  \hskip 5mm \forall p < n-3 $

  $ [ H^{p}_{\m}(H_{L+K}^{r}(R))]_{_{\bf -1}} \cong [ H^{p}_{\m}(H_{M+N}^{r+1}(R))]_{_{\bf -1}} \hskip 5mm \forall p, \forall r >3.$

\vskip 2mm
Now, if we go back to the long exact sequence \ref{long} considered in 
Discussion \ref{MVsplit} to compute the Lyubeznik table of $R/J(G)$ we get
{\small $$
 0 \longrightarrow   \K^{\lambda_{n-3,n-2}(R/L)} \longrightarrow \K^{\lambda_{n-3,n-2}(R/J(G))} \longrightarrow \K^{\lambda_{n-3,n-3}(R/L+K)}  \xrightarrow{\partial^3_{n-3}}   \K \oplus \K  \longrightarrow  \K \longrightarrow 0
$$}
\noindent and $\lambda_{p,n-2}(R/L)= \lambda_{p,n-2}(R/J(G))$ for all $p < n-3 $. Moreover, for  $r>2$ and any $p$, the long exact sequence 
becomes 
{\small \begin{align*}
 \cdots \longrightarrow & \K^{\lambda_{p-1,n-(r+1)}(R/L+K)}  \xrightarrow{\partial^r_{p-1}} \K^{\lambda^r_{p,n-r}(R/L)} \longrightarrow \K^{\lambda_{p,n-r}(R/J(G))} \longrightarrow \K^{\lambda_{p,n-(r+1)}(R/L+K)}  \xrightarrow{\partial^r_{p}}   \cdots
\end{align*}}
Therefore we have:

 $ \lambda_{d,d}(R/J(G))=1,$
 
 $\lambda_{d-1,d}(R/J(G))= \lambda_{d-1,d}(R/L)+1,$
 
 $\lambda_{p,d}(R/J(G))= \lambda_{p,d}(R/L)  \hskip 5mm \forall p < d-1$

\noindent and the rest of Lyubeznik numbers depend on the connecting morphisms $$\K^{\lambda_{p-1,n-(r+2)}(R/M+N)}=\K^{\lambda_{p-1,n-(r+1)}(R/L+K)}  \xrightarrow{\partial^r_{p-1}} \K^{\lambda_{p,n-r}(R/L)}.$$

\noindent Assuming that the connecting morphisms are zero we get 
$${\lambda_{p,n-r}(R/J(G))} = {\lambda_{p,n-r}(R/L)} + {\lambda_{p,n-(r+1)}(R/L+K)}  $$
and the result follows since we have $L=J(G\setminus \{x_n\})$ and 
an isomorphism
$$\K[x_1,\dots, x_n]/M+N \cong \K[x_1,\dots, x_{n-3}]/J(H)$$
and thus $\Lambda(R/M+N) = \Lambda(\K[x_1,\dots, x_{n-3}]/J(H))$.
\end{proof}

We highlight the following particular case.

\begin{corollary}\label{deg_two_2}
Under the assumptions of Proposition \ref{deg_two}, if $\Lambda(R/J(G\setminus \{x_n\}))$ is trivial then the Lyubeznik table of $R/J(G)$ is 
$$\Lambda(R/J(G))  = \left(
                    \begin{array}{ccccccc}
                     0& 0& \cellcolor{usc!60}\la'_{0,0} &\cellcolor{usc!60} \la'_{0,1} & \cellcolor{usc!60}\cdots & \cellcolor{usc!60}\la'_{0,d-3}& 0  \\
                     & 0&  0 &  \cellcolor{usc!60}\la'_{1,1}& \cellcolor{usc!60}\cdots & \cellcolor{usc!60}\la'_{1,d-3}& 0 \\
                     & &  0 &  0& \cellcolor{usc!60}\cdots & \cellcolor{usc!60}\la'_{2,d-3}& 0 \\
                     & &  && \cellcolor{usc!60}\ddots & \cellcolor{usc!60}\vdots & \vdots\\
                      & & &  &  & \cellcolor{usc!60}\la'_{d-3,d-3} & 0 \\                
                     & & &  &  & 0 & 0 \\
                     & & &  &  & 0 & \cellcolor{usc!40}1 \\
                     & & & &  &  & \cellcolor{usc!40} 1 \\
                    \end{array}
                  \right)
$$ where $$\Lambda(\K[x_1,\dots, x_{n-3}]/J(H))  = \left(
                    \begin{array}{ccc}
                      \cellcolor{usc!60}\la'_{0,0} & \cellcolor{usc!60}\cdots & \cellcolor{usc!60}\la'_{0,d-3}  \\
                       & \cellcolor{usc!60}\ddots & \cellcolor{usc!60}\vdots \\
                       &  & \cellcolor{usc!60}\la'_{d-3,d-3} \\
                    \end{array}
                  \right)
$$

\end{corollary}

\subsubsection{Splitting vertices of maximal degree}
We turn our attention to the case of a splitting vertex $x_n$ of degree $n-1$. That is, $\{x_i,x_n\}$ is an edge of $G$ for $i=1,\dots, n-1$.

\begin{proposition} \label{dominating}
Let $J(G) \subseteq R$ be the cover ideal of a simple connected graph $G$. Let $x_n\in V_G$ be a vertex of degree $n-1$. Then, the 
Lyubeznik table of $R/J(G)$ is trivial.
\end{proposition} 

\begin{proof}
We have a MV-splitting $J_{G}= L\cap K $ where $ K=(x_{1}, x_{n}) \cap \cdots \cap ( x_{n-{2}} , x_{n}) $ and, given the fact that all the vertices are in the neighbourhood of $x_n$, we have
 { $$
  L+K = \bigcap_{\substack{i,j\neq n \\ \{x_i,x_j\}\in E_G}}(x_i,x_j,x_n)  = L+(x_n)$$}
Recall that  $\Lambda(R/K)$ is trivial so the long exact sequence associated to the MV-splitting reduces to
 $$ \cdots  \longrightarrow H^{p}_{\m}(H_{J(G)}^{r}(R)) \longrightarrow H^{p}_{\m}(H_{L+(x_n)}^{r+1}(R))  \xrightarrow{\partial^r_p}  
  H^{p+1}_{\m}(H_{L}^{r}(R))   \longrightarrow \cdots$$ 
Using the interpretation of the connecting morphisms $\partial^r_p$'s in terms of the corresponding linear strands given at the end of Section \ref{local_cohomology}  we observe that we are 
%
%
comparing the linear strands $\mathbb{F}_{\bullet}^{<r+1>}((L+(x_n))^{\vee})^{\ast}$ and $\mathbb{F}_{\bullet}^{<r>}(L^{\vee})^{\ast}$ which are essentially the same (modulo a shifting), so the induced morphisms in homology are  isomorphisms and the result follows.
\end{proof}

\subsection{Examples} \label{Lyubeznik3}
The MV-splitting techniques developed in the previous section allow us to compute the Lyubeznik table of many families of graphs directly from the combinatorics of the graph without an explicit computation of the corresponding local cohomology modules. The idea is to choose a convenient splitting vertex and reduce the computation to the case of a graph in a smaller number of vertices.

An interpretation of  Proposition \ref{whisker} is that the Lyubeznik table remains  invariant under the operation of removing whiskers. 
In this way we may simplify our original graph and, in the case of acyclic graphs, we can deduce the triviality of the Lyubeznik table by reducing the computation to the case of a single edge. In particular we get:

\begin{corollary}
The Lyubeznik table of the cover ideal of a {\bf path} is trivial.
\end{corollary}

\begin{corollary}
The Lyubeznik table of the cover ideal of a {\bf tree} is trivial.
\end{corollary}

Using Corollary \ref{disjoint_trivial} we can also consider the case of forests.

\begin{corollary}
The Lyubeznik numbers of the cover ideal of a {\bf forest} with $c$ connected components are 
$$\lambda_{d-2k,d-k}(R/J(G))= {c \choose k+1} \hskip 5mm {\rm for} \hskip 5mm k=0,\dots , c -1$$
and the rest of Lyubeznik numbers are zero..
\end{corollary}

Another source of examples of trivial Lyubeznik tables is using Proposition \ref{dominating}. It says that the {\bf cone} of any graph $G$, has a trivial Lyubeznik table. In particular:

\begin{corollary}
The Lyubeznik table of the cover ideal of a {\bf wheel} is trivial.
\end{corollary}

Another way of simplifying our original graph is by means of Proposition \ref{handle}, which says 
 that we can remove what we call {\bf handles} (or equivalently $3$ and $4$-cycles) having the following shape:
 
\vskip -2cm \begin{center}
\includegraphics[width=100mm,scale=0.5]{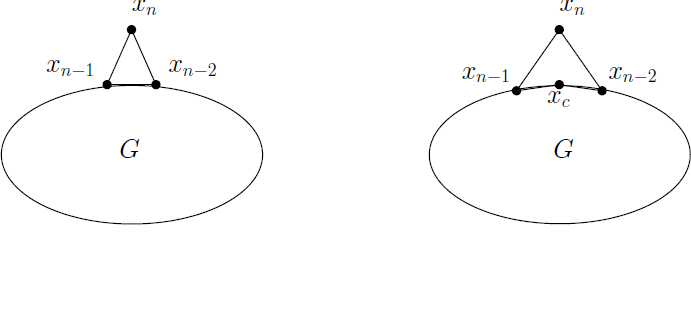}
\end{center}

\vskip -1cm This gives as a very visual method to reduce the computation of Lyubeznik tables of graphs. For example, removing the yellow vertices indicated below do not modify the Lyubeznik table and thus, in the end, we see that the Lyubeznik table of the following graph is trivial.
 
\begin{center}
\includegraphics[width=120mm,scale=0.5]{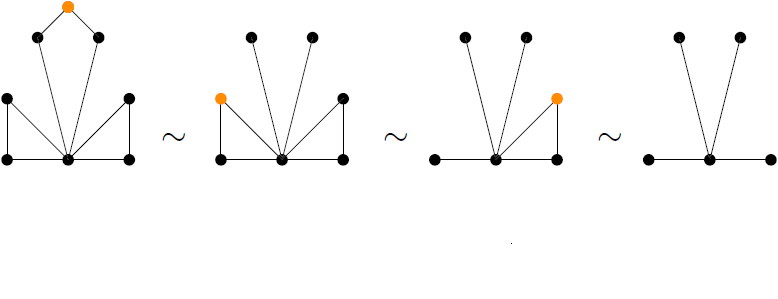}
\end{center}

\vskip -1.5cm Now we turn our attention to the case of cycles. Applying  iteratively Corollary \ref{deg_two_2}  we will obtain a closed formula for the Lyubeznik numbers.  To illustrate our methods we present  the case of a $6$-cycle. The corresponding graphs $L$ and $H$ are represented  as follows:
\begin{center}
\includegraphics[width=100mm,scale=0.3]{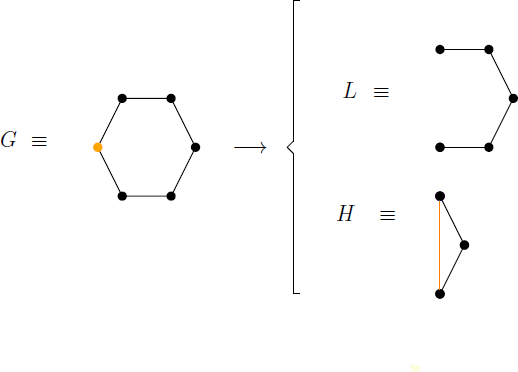}
\end{center}

Notice that $L$ is a path so it has trivial Lyubeznik table and thus we can use Corollary \ref{deg_two_2}. Moreover we have that $H$ is a $3$-cycle.


\begin{proposition} \label{cycle}
The Lyubeznik numbers of the cover ideal of a {\bf $n$-cycle} with $n=3k+\ell$, $\ell\in\{-1,0,1\}$, are 

$\lambda_{d-3i,d-i}(R/J(G))= 1, \hskip 5mm   \lambda_{d-3i-1,d-i}(R/J(G))= 1  \hskip 5mm {\rm for} \hskip 5mm i=0,\dots , k -2. $

$\lambda_{d-3i,d-i}(R/J(G))= 1   \hskip 7mm {\rm for} \hskip 5mm i=k -1. $

\noindent and the rest of Lyubeznik numbers are zero.
\end{proposition}

\begin{proof}
Notice that, if we denote $G$ the $n$-cycle, then $G\setminus \{x_n\}$ is a path so its Lyubeznik table is trivial. On the other hand, the graph $H$ obtained from $G$ is a $(n-3)$-cycle so we can use induction on $k$ and Corollary \ref{deg_two_2} to produce the formula for the Lyubeznik numbers.
\end{proof}

\begin{example}
The Lyubeznik table of a cycle $C_n$ in  $n$ vertices  for  $n=5,\dots , 11$ are:
{\tiny $$
                  \left(
                    \begin{array}{cccc}
                     0& 0& \cellcolor{usc!50}1&0\\
                      & 0& 0&0\\
                      & & 0&\cellcolor{usc!50}1\\
                      & & &\cellcolor{usc!50}1\\
                    \end{array}
                  \right)
\hskip .5cm   \left(
                    \begin{array}{ccccc}
                    0& 0 & 0 & 0 & 0  \\
                     &0& 0& \cellcolor{usc!50}1&0\\
                     & & 0& 0&0\\
                      && & 0&\cellcolor{usc!50}1\\
                      && & &\cellcolor{usc!50}1\\
                    \end{array}
                  \right)
\hskip .5cm  \left(
                    \begin{array}{cccccc}
                    0& 0& 0 & 0 & 0 & 0  \\
                    &0& 0 & 0 & 0 & 0  \\
                    & &0& 0& \cellcolor{usc!50}1&0\\
                    & & & 0& 0&0\\
                    &  && & 0&\cellcolor{usc!50}1\\
                     & && & &\cellcolor{usc!50}1\\
                    \end{array}
                  \right)
\hskip .5cm \left(
                    \begin{array}{ccccccc}
                      0& 0& 0&0 & \cellcolor{usc!50}1 & 0& 0  \\                                    
                       & 0& 0& 0&0 & 0 & 0 \\
                       &  & 0& 0& 0&\cellcolor{usc!50}1 & 0\\
                      &  &  & 0& 0& \cellcolor{usc!50}1&0\\
                      &  &  & & 0& 0&0\\
                      &  &  & & & 0&\cellcolor{usc!50}1\\
                      &  &  & & & &\cellcolor{usc!50}1\\
                    \end{array}
                  \right)$$ $$\left(
                    \begin{array}{cccccccc}
                    0& 0& 0& 0& 0 & 0 & 0 & 0  \\
                    &  0& 0& 0&0 & \cellcolor{usc!50}1 & 0& 0  \\                                    
                     &  & 0& 0& 0&0 & 0 & 0 \\
                     &  &  & 0& 0& 0&\cellcolor{usc!50}1 & 0\\
                    &  &  &  & 0& 0& \cellcolor{usc!50}1&0\\
                     & &  &  & & 0& 0&0\\
                    &  &  &  & & & 0&\cellcolor{usc!50}1\\
                    &  &  &  & & & &\cellcolor{usc!50}1\\
                    \end{array}
                  \right) \hskip .5cm \left(
                    \begin{array}{ccccccccc}
                    0&0& 0& 0& 0& 0 & 0 & 0 & 0  \\
                   & 0& 0& 0& 0& 0 & 0 & 0 & 0  \\
                   & &  0& 0& 0&0 & \cellcolor{usc!50}1 & 0& 0  \\                                    
                    & &  & 0& 0& 0&0 & 0 & 0 \\
                    & &  &  & 0& 0& 0&\cellcolor{usc!50}1 & 0\\
                   & &  &  &  & 0& 0& \cellcolor{usc!50}1&0\\
                    & & &  &  & & 0& 0&0\\
                   & &  &  &  & & & 0&\cellcolor{usc!50}1\\
                   & &  &  &  & & & &\cellcolor{usc!50}1\\
                    \end{array}
                  \right) \hskip .5cm   
                  \left(
                    \begin{array}{cccccccccc}
                      0& 0& 0& 0& 0& 0& \cellcolor{usc!50}1 & 0  & 0 & 0  \\
                      & 0& 0& 0&0& 0& 0 & 0 & 0 & 0  \\
                      & & 0& 0& 0& 0& 0 & \cellcolor{usc!50}1 & 0& 0  \\
                      & & & 0& 0& 0&0 & \cellcolor{usc!50}1 & 0& 0  \\                                    
                     & & &  & 0& 0& 0&0 & 0 & 0 \\
                     & & &  &  & 0& 0& 0&\cellcolor{usc!50}1 & 0\\
                     & & & &  &  & 0& 0& \cellcolor{usc!50}1&0\\
                     & & & &  &  & & 0& 0&0\\
                     & & & &  &  & & & 0&\cellcolor{usc!50}1\\
                     & & & &  &  & & & &\cellcolor{usc!50}1\\
                    \end{array}
                  \right)$$}
\end{example}
Certainly  a new source for finding non trivial Lyubeznik tables is to consider graphs obtained by joining  cycles. After removing whiskers and handles we reduce to the case of cycles joined by paths or sharing edges  in such a way that we can still find  degree two vertices that we can remove in order to simplify the graph. For example, the complement of a $6$-cycle

\begin{center}
\includegraphics[width=40mm,scale=0.5]{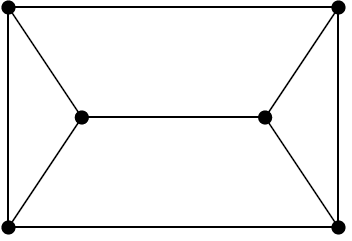}
\end{center}


\noindent can be interpreted as joining two $3$-cycles and two $4$-cycles. However all the vertices have degree three so we cannot apply our methods.

To start our study we will consider the following joining operation.

\begin{definition} \label{L-join}
Let $C_m$ and $C_n$ be two cycles in $m$ and $n$ vertices respectively. We say that they are $L$-joined and we will denote by $C_m \divideontimes C_n$ the corresponding graph if they share at most one edge or if they are  joined by a path. 
\end{definition}

When we remove a vertex from $C_m \divideontimes C_n$ we do not obtain a subgraph $L$ with trivial Lyubeznik table. However, we still can use 
Proposition \ref{deg_two} to obtain the following formula for the Lyubeznik numbers. 

\begin{proposition} \label{two_cycles}
Let $C_m$ and $C_n$ be two cycles  with  $m\leq n$ and $m=3k_1+\ell_1$, $n=3k_2+\ell_2$ and $\ell_1, \ell_2 \in\{-1,0,1\}$. 
Then the Lyubeznik numbers of the cover ideal of $G= C_m \divideontimes C_n$  are 

$\lambda_{d-3i,d-i}(R/J(G))= i+1, \hskip 5mm   \lambda_{d-3i-1,d-i}(R/J(G))= i+2  \hskip 5mm {\rm for} \hskip 5mm i=0,\dots , k_1 -2, $

$\lambda_{d-3i,d-i}(R/J(G))= k_1, \hskip 9mm   \lambda_{d-3i-1,d-i}(R/J(G))= k_1  \hskip 5mm {\rm for} \hskip 5mm i=k_1-1,\dots , k_2 -2,$

$\lambda_{d-3i,d-i}(R/J(G))= k_1+k_2-i-1, \hskip 5mm   \lambda_{d-3i-1,d-i}(R/J(G))= k_1+k_2-i-2  \hskip 5mm $

\noindent ${\rm for} \hskip 3mm i=k_2-1,\dots , k_1+k_2 -2$
and the rest of Lyubeznik numbers are zero. Here we follow the convention that in the case where $k_1=k_2$ the set of indices $i=k_1-1,\dots , k_2 -2$ is empty.
\end{proposition}

\begin{proof}
We are going to use Proposition \ref{deg_two} so we will follow the same terminology used in its proof. 
Pick a splitting vertex $x_m$ of degree  two in $C_m$. Notice that $G\setminus \{x_m\}$ has the same Lyubeznik table as $C_n$ and the graph $H$ associated to $G$ is $C_{m - 3} \divideontimes C_n$ so we can proceed by induction on $k_1$.  The fact that $C_m$ and $C_n$ share at most one edge implies that the induction step will end up in a graph $H$ having the Lyubeznik table of $C_n$.  
\end{proof}

We illustrate the case  $G=C_6 \divideontimes C_6$ as follows:

\begin{center}
\includegraphics[width=110mm,scale=0.7]{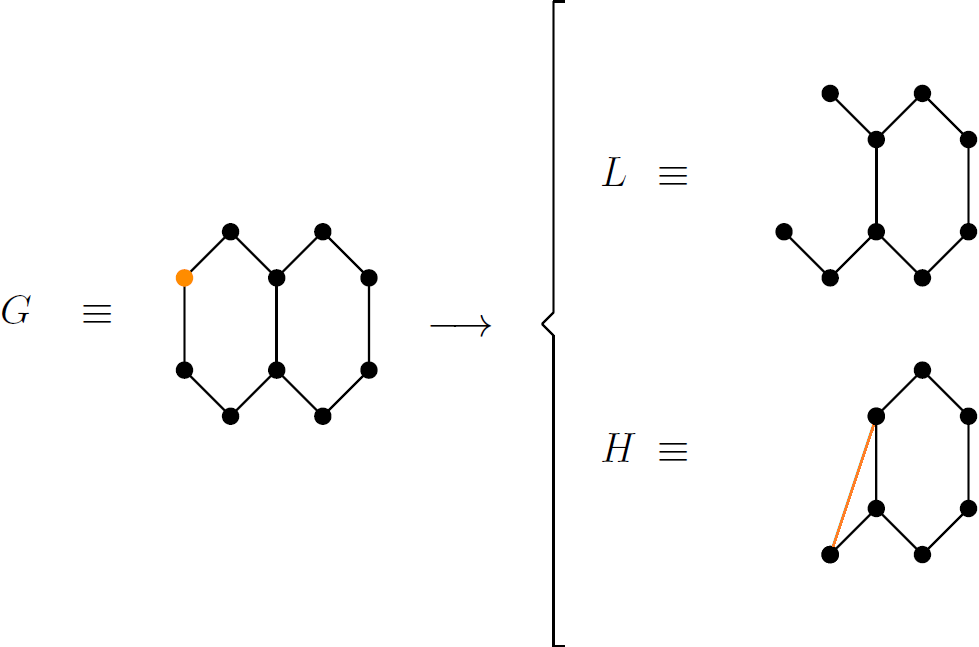}
\end{center}


More generally we can consider the family of graphs obtained using the $\divideontimes$ operation, which is a family that includes {\bf cycles with chords} or {\bf cactus graphs}. We can iterate the methods used in Proposition \ref{two_cycles} in order to compute the Lyubeznik table of any $L$-joined graph of the form $C_{n_1}\divideontimes \cdots \divideontimes C_{n_r}$. We are not going to give a closed formula for its Lyubeznik table  since it depends on the indices $n_i=3k_i+\ell_i$, with $\ell_i \in\{-1,0,1\}$, for $i=1,\dots, r$  and one should distinguish too many cases which makes it very tedious and not very illustrative.

We will just point out that the non-vanishing Lyubeznik numbers are 

$\lambda_{d-3i,d-i}(R/J(G)), \hskip 5mm\lambda_{d-3i-1,d-i}(R/J(G))$ \hskip 5mm for $i=0,\dots, k_1+\cdots + k_r - r$ 

 Moreover we have  $\lambda_{d,d}= 1$, $\lambda_{d-1,d}=r$ and $\lambda_{d-3(k_1+\cdots + k_r - r),d-(k_1+\cdots + k_r - r)}=1$.

\vskip 2mm

In the case that we have two cycles $C_m$ and $C_n$ sharing more than one edge we can still use iteratively  Proposition \ref{deg_two} to compute its Lyubeznik table. However, in this case we will not end up with a graph $H$ having the Lyubeznik table of $C_n$ as the following example shows.

\begin{example}  \label{cycle_sharing} 
Let $G$ be two $8$-cycles sharing four vertices. The Lyubeznik table of $G$ can be obtained, using Proposition \ref{deg_two}, from the Lyubeznik tables of two graphs $L_1$ and $H_1$. Notice that $L_1$ has the same Lyubeznik type  as an $8$-cycle $C_8$, but in order to get the Lyubeznik table of $H_1$ we have to apply  Proposition \ref{deg_two}
once again.  We illustrate the procedure as follows:
 
\begin{center}
\includegraphics[width=150mm,scale=0.3]{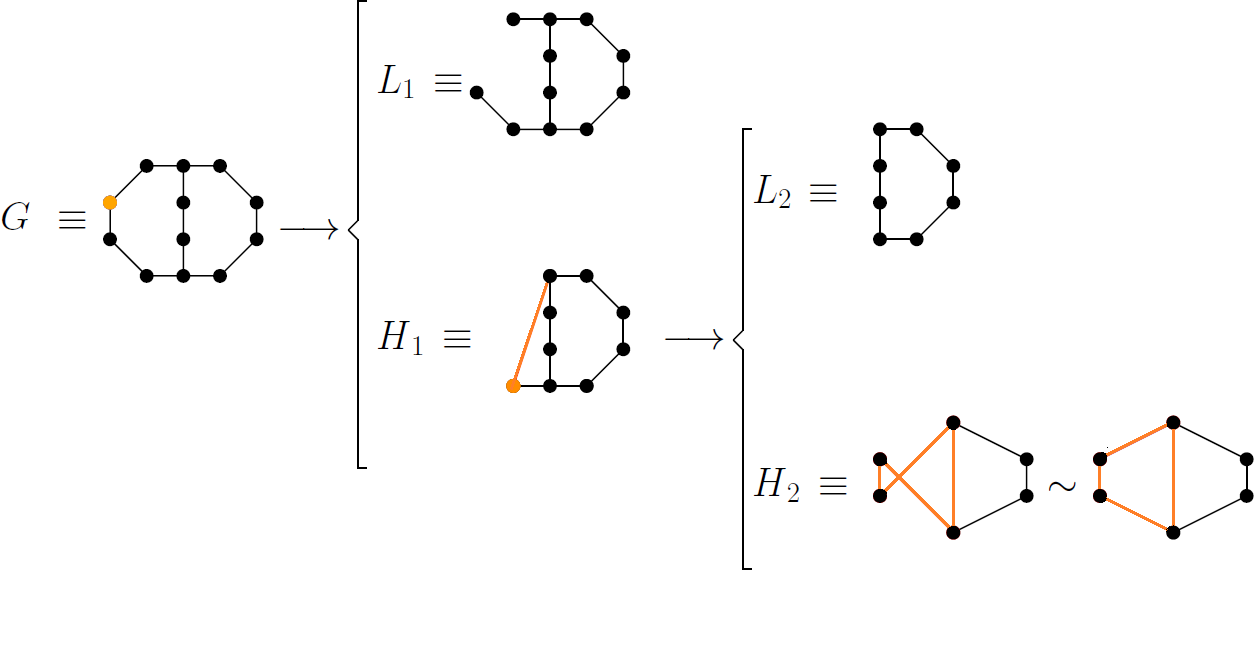}
\end{center}

We have that $L_2$ is an $8$-cycle so we know its Lyubeznik table by Proposition \ref{cycle}. On the other hand, 
 $H_2$ is a graph in $6$ vertices whose Lyubeznik table is trivial by using Proposition \ref{handle}.  Therefore we can use Proposition \ref{deg_two} to compute the Lyubeznik table of $H_1$

$$ 
 \left(
                    \begin{array}{cccccccc}
                     0&  0& 0&  0&  0 &  0 &  0 &  0  \\
                     &  0& 0&  0&  0 & \cellcolor{usc!50}1 &  0 &  0  \\
                     & &  0&  0& 0 &  0 &  0 & 0  \\
                     & & &  0&  0 &  0 & \cellcolor{usc!50}1&  0  \\
                     & & & & 0 & 0 &\cellcolor{usc!50} 1& 0  \\                                    
                     && & &  &  0 & 0 &  0 \\
                     && & &  &  &  0 & \cellcolor{usc!50}1 \\
                     && & & &  &  & \cellcolor{usc!50}1 \\
                    \end{array}
                  \right) 
\hskip .5cm ,    \hskip .5cm               
                 \left(
                    \begin{array}{ccccc}
                     \cellcolor{usc!80}0&\cellcolor{usc!80} 0&\cellcolor{usc!80} 0 & \cellcolor{usc!80}0 &\cellcolor{usc!80} 0  \\
                     & \cellcolor{usc!80} 0& \cellcolor{usc!80} 0 &  \cellcolor{usc!80}0 &\cellcolor{usc!80} 0 \\
                     & & \cellcolor{usc!80} 0 &  \cellcolor{usc!80}0 & \cellcolor{usc!80}0 \\
                      &  &  & \cellcolor{usc!80}0 &\cellcolor{usc!80} 0 \\
                      & &  &  & \cellcolor{usc!80}1 \\
                    \end{array}
                  \right)
\hskip .5cm   \rightsquigarrow  \hskip .5cm 
\left(
                    \begin{array}{cccccccc}
                     \cellcolor{usc!20}0& \cellcolor{usc!20} 0& \cellcolor{usc!80} 0&\cellcolor{usc!80} 0& \cellcolor{usc!80}0 &\cellcolor{usc!80} 0 & \cellcolor{usc!80}0 & \cellcolor{usc!20}0  \\
                     & \cellcolor{usc!20}0& \cellcolor{usc!20}0& \cellcolor{usc!80}0&\cellcolor{usc!80} 0 & \cellcolor{usc!80}1 & \cellcolor{usc!80}0 & \cellcolor{usc!20}0  \\
                     & & \cellcolor{usc!20}0&\cellcolor{usc!20} 0& \cellcolor{usc!80}0 &\cellcolor{usc!80} 0 &\cellcolor{usc!80} 0 &\cellcolor{usc!20} 0  \\
                     & & & \cellcolor{usc!20}0& \cellcolor{usc!20}0 & \cellcolor{usc!80}0 & \cellcolor{usc!80}1&\cellcolor{usc!20} 0  \\
                     & & & & \cellcolor{usc!20}0 & \cellcolor{usc!20}0 &\cellcolor{usc!80} 2& \cellcolor{usc!20}0  \\                                    
                     && & &  &\cellcolor{usc!20} 0 & \cellcolor{usc!20}0 & \cellcolor{usc!20}0 \\
                     && & &  &  & \cellcolor{usc!20}0 & \cellcolor{usc!50} 2 \\
                     && & & &  &  & \cellcolor{usc!50}1 \\
                    \end{array}
                  \right)
$$ 

From the Lyubeznik table of $H_1$ given above and the one for $L_1$ we deduce the Lyubeznik table of $G$.

$$
                  \left(
                    \begin{array}{ccccccccccc}
                     0& 0& 0& 0& 0& 0& 0& 0 & 0 & 0 & 0  \\
                     & 0& 0& 0& 0& 0& 0& 0 & 0  & 0 & 0  \\
                     & & 0& 0& 0&0& 0& 0 & 0 & 0 & 0  \\
                     & & & 0& 0& 0& 0& 0 & 0 & 0& 0  \\
                     & & & & 0& 0& 0&0 & \cellcolor{usc!50}1 & 0& 0  \\                                    
                     && & &  & 0& 0& 0&0 & 0 & 0 \\
                     && & &  &  & 0& 0& 0&\cellcolor{usc!50}1 & 0\\
                     && & & &  &  & 0& 0& \cellcolor{usc!50}1&0\\
                     && & & &  &  & & 0& 0&0\\
                     && & & &  &  & & & 0&\cellcolor{usc!50}1\\
                     && & & &  &  & & & &\cellcolor{usc!50}1\\
                    \end{array}
                  \right)
\hskip .5cm   \rightsquigarrow  \hskip .5cm 
\left(
                    \begin{array}{ccccccccccc}
                     0& 0& \cellcolor{usc!20}0& \cellcolor{usc!20}0&\cellcolor{usc!80} 0&\cellcolor{usc!80} 0&\cellcolor{usc!80} 0&\cellcolor{usc!80} 0 &\cellcolor{usc!80} 0 & \cellcolor{usc!20}0 & 0  \\
                     & 0& 0& \cellcolor{usc!20}0&\cellcolor{usc!20} 0&\cellcolor{usc!80} 0&\cellcolor{usc!80} 0& \cellcolor{usc!80}1 &\cellcolor{usc!80} 0  & \cellcolor{usc!20}0 & 0  \\
                     & & 0& 0&\cellcolor{usc!20} 0&\cellcolor{usc!20}0& \cellcolor{usc!80}0& \cellcolor{usc!80}0 &\cellcolor{usc!80} 0 &\cellcolor{usc!20} 0 & 0  \\
                     & & & 0& 0&\cellcolor{usc!20} 0& \cellcolor{usc!20}0& \cellcolor{usc!80}0 & \cellcolor{usc!80}1 & \cellcolor{usc!20}0& 0  \\
                     & & & & 0& 0&\cellcolor{usc!20} 0&\cellcolor{usc!20}0 &\cellcolor{usc!80} 3 & \cellcolor{usc!20}0& 0  \\                                    
                     && & &  & 0& 0& \cellcolor{usc!20}0&\cellcolor{usc!20}0 &\cellcolor{usc!20} 0 & 0 \\
                     && & &  &  & 0& 0&\cellcolor{usc!20} 0&\cellcolor{usc!50}3 & 0\\
                     && & & &  &  & 0& 0& \cellcolor{usc!50}2&0\\
                     && & & &  &  & & 0& 0&0\\
                     && & & &  &  & & & 0&\cellcolor{usc!50}2\\
                     && & & &  &  & & & &\cellcolor{usc!50}1\\
                    \end{array}
                  \right)
$$
\end{example}

\vskip 5mm

\begin{remark}
The number  of connected graphs in a given number of vertices is very large so it seems infeasible to
 find all the possible configurations of Lyubeznik tables.  If the number of vertices is small, we can use our methods to reduce the number of examples  
 we have to consider and compute the Lyubeznik numbers of the remaining graphs using the algorithm presented in \cite{AF13}.  

\vskip 2mm

We made the computations in the case we have either $4,5$ or $6$ vertices and the possible Lyubeznik types are reflected in the following table.
Not surprisingly, having trivial Lyubeznik table is the most common situation.

\def\arraystretch{1.3}
\begin{table}[!ht] \label{table:T2}
\centering
\begin{tabular}{|c|c|c|c|c|}
\hline
\(n\) & \( {\tt trivial} \) & \({\tt cycle} \) & \({\tt {complement \hskip 1mm cycle}} \)& \({\tt total} \)\\[1.5pt] \hline
\hline
\(4\) & \( 6\) & \( - \)  & \( - \) & \( 6\) \\[1.5pt] \hline
\(5\) & \(20\) & \( 1 \)  & \( - \)  & \( 21\)\\[1.5pt] \hline
\(6\) & \(106\) & \( 5 \)  & \(  1\) & \( 112\) \\[1.5pt] \hline
\end{tabular}
\vspace{5pt}
\caption{ Number of graphs in  $n$ vertices with given Lyubeznik type.}
\end{table}

In order to expand this list  it would be desirable to develop techniques to deal with splitting vertices of degree bigger than two and 
identify graphs with non trivial Lyubeznik tables other than cycles or complement of cycles. 
\end{remark}

\section{Bass numbers of local cohomology modules of cover ideals of graphs} \label{Bass}

Let $G$ be a simple graph and, given $\alpha \in \{0,1\}^n$, we denote by $G_\alpha$ the subgraph of $G$ obtained by removing the vertices $x_i$ such that $\alpha_i=0$.
Indeed we will only consider those $\alpha$'s for which $G_\alpha$ is not a set of isolated vertices, which means that $E_{G_\alpha}\neq \emptyset$. Notice that the cover ideal of $G_\alpha$
is an ideal in the polynomial ring $R_{\fp_\alpha}=\K[x_i \hskip 2mm | \hskip 2mm\alpha_i=1].$
As we mentioned in Remark  \ref{Bass_restriction}, Bass numbers behave well with respect to restriction and thus, the Bass numbers of the local cohomology module 
$H_{J(G)}^{n-i}(R)$ with respect to the face ideal $\fp_\alpha$ are nothing but the Lyubeznik numbers corresponding to the subgraph $G_\alpha$. More precisely,
$$\mu_p(\fp_\alpha, H_{J(G)}^{n-i}(R) )= \lambda_{p,i} (R_{\fp_\alpha}/J(G_\alpha))$$
Certainly  $G_\alpha$ are not necessarily connected graphs and thus one needs to  use Proposition \ref{dual_join}  in order to compute these Lyubeznik numbers. 
%
%
In what follows we will denote $c_\alpha$ as the number of connected components of $G_\alpha$ and 
$c_{max}=\max \{ c_\alpha \hskip 2mm | \hskip 2mm \alpha\in \{0,1\}^n\}$.

\begin{example} \label{cmax}
The maximal number of connected components of a path or a cycle is achieved when we remove every third vertex from the graph. 
For an $n$-path we have $c_{max}= \lceil \frac{n+1}{3} \rceil$. However for $n$-cycle we have to be a little more careful and we have $c_{max}= \lfloor \frac{n}{3} \rfloor$. In particular, for $n=3k - 1$ we have that the $n$-path has $c_{max}=k$ and the $n$-cycle has $c_{max}=k-1$. For the rest of cases they coincide.
\end{example}

\subsection{Linear injective resolutions} \label{Bass1}
Let $G$ be a graph that $R/J(G)$ is Cohen-Macaulay. The Bass numbers of this class of graphs are completely determined as it has been shown in \cite{AV14} using a simple spectral sequence argument. For completeness we include the result here but giving an equivalent description in terms of the connectivity of the subgraphs. 

\begin{proposition}
Let $J(G) \subseteq R$ be the cover ideal of a simple connected graph $G$. Then the following are equivalent
\begin{itemize}
\item[i)] $R/J(G)$ is Cohen-Macaulay.
\item[ii)] $\mu_p(\fp_\alpha, H^2_{J(G)}(R))= \delta_{p,|\alpha|-2}$
\item[iii)] $G_\alpha$ is connected and $\Lambda(R_{\fp_\alpha}/J(G_\alpha))$ is trivial for all $\alpha \in \{0,1\}^n$.
\end{itemize}
\end{proposition}

In this case, the $\bZ^n$-graded injective resolution  
$\mathbb{I}_{\bullet}(H_{J(G)}^2(R)) $ has a very rigid structure which resembles the injective resolution of Gorenstein rings. Namely we have:
$$ \xymatrix{ 0 \ar[r]& H_{J(G)}^2(R)\ar[r] &\underset{|\alpha| = 2}{\bigoplus} 
E_\alpha
\ar[r]&  \underset{|\alpha| = 3}{\bigoplus} 
E_\alpha \ar[r]&\cdots \ar[r]&   \underset{|\alpha| = n-1}{\bigoplus}
E_\alpha \ar[r]&
E_{\bf 1} \ar[r]&
0},$$ 
where, at each component of the resolution we are only considering those $\alpha$'s such that $\fp_\alpha \in \Supp  R/J(G)$.
Notice that this resolution is linear, it only has one linear strand. 

%

We will show next that we may find non Cohen-Macaulay graphs  still having a rigid injective resolution.  To such purpose let's consider the following family.

\begin{assumptions} \label{assump}
Let $J(G) \subseteq R$ be the cover ideal of a simple (not necessarily connected) graph $G$ such that, for any $\alpha\in \{0,1\}^n$,  all the connected components of $G_\alpha$ have trivial Lyubeznik table. 
\end{assumptions}

As we have seen in Section \ref{Lyubeznik}, the condition of having trivial Lyubeznik table is very common and is not difficult to find families of 
graphs satisfying Assumptions \ref{assump}. An interesting example would be the case of {\bf forests} but we may also include Cohen-Macaulay graphs. As a direct consequence of Proposition \ref{dual_join} we get:

\begin{theorem} \label{Bass_trivial}
Let $J(G) \subseteq R$ be an ideal satisfying Assumptions \ref{assump}. Then, the Bass numbers of the corresponding local cohomology modules are
$$\mu_p(\fp_\alpha, H^{k+1}_{J(G)}(R))= \delta_{p,|\alpha|-2k} \cdot {c_\alpha \choose k} \hskip 5mm {\rm for} \hskip 5mm k=1,\dots , c_\alpha $$
\end{theorem}

In this case we also have a rigid injective resolution in the sense that they only have one linear strand.

\begin{example}
Let $G$ be a $5$-path. We have that all the subgraphs $G_\alpha$ are connected except for the case $\alpha=(1,1,0,1,1)$ in which $G_\alpha$ has two connected components. Applying the above result we get the following linear injective resolutions

$ \xymatrix{ 0 \ar[r]& H_{J(G)}^2(R)\ar[r] &\underset{|\alpha| = 2}{\bigoplus}
E_\alpha
\ar[r]&  \underset{|\alpha| = 3}{\bigoplus}
E_\alpha \ar[r]&  {\begin{array}{c} E_{(1,1,0,1,1)}^2  \\ \oplus  \\ \left( \underset{\substack{|\alpha| = 4 \\ \alpha \neq (1,1,0,1,1)} }{\bigoplus}
E_\alpha \right) \end{array}}\ar[r]&
E_{\bf 1} \ar[r]&
0},$

$ \xymatrix{ 0 \ar[r]& H_{J(G)}^3(R)\ar[r] &  E_{(1,1,0,1,1)} \ar[r]&
0},$

\noindent where we are only considering those $\alpha$'s such that $\fp_\alpha \in \Supp  R/J(G)$. Notice that in this case we have two local cohomology modules different from zero.
\end{example}

Quite surprisingly we may provide a vanishing criterion for the local cohomology modules in terms of the connected components of the subgraphs. Recall that the {\bf cohomological dimension} of an ideal $J$ is the maximum $r$ for 
which $H_J^r(R)\neq 0$.

\begin{proposition}\label{max}
Let $J(G) \subseteq R$ be an ideal satisfying Assumption \ref{assump}. Then $$H^{k+1}_{J(G)}(R)) \neq 0  \hskip 5mm {\rm for} \hskip 5mm  k=1, \dots , c_{max}.$$ In particular, the cohomological dimension of $J(G)$ is 
$${\rm cd}(J(G), R)= c_{max}+1.$$
\end{proposition}

\begin{proof}
A local cohomology module is different from zero if it has a non-vanishing Bass number. Then the result follows from Theorem \ref{Bass_trivial}.
\end{proof}

Using  the relation between the cohomological dimension and the projective dimension of the Alexander dual ideal  given in \cite{ER98} we deduce the following result which, in particular, gives a very simple description of the projective dimension of edge ideals of forests (compare with the results  in \cite{JK05}).

\begin{proposition}
Let $J(G) \subseteq R$ be an ideal satisfying Assumption \ref{assump}. Then, the projective dimension of the corresponding edge ideal $I(G)$ is 
$${\rm pd}(R/I(G))=c_{max}+1.$$ 
\end{proposition}

%
%

\subsection{Non linear injective resolutions} \label{Bass2}
As we have seen through the examples in Section \ref{Lyubeznik}, the easiest way to find non linear injective resolutions is to consider the case of cycles.
Let's illustrate this fact with the following examples

\begin{example}
Let $G$ be a $6$-cycle. We have that all the subgraphs $G_\alpha$ are connected paths except for the cases $\alpha=(1,1,0,1,1,0), (1,0,1,1,0,1),(0,1,1,0,1,1)$ in which $G_\alpha$ has two connected components. Therefore we have the injective resolutions

{\small$ \xymatrix{ 0 \ar[r]& H_{J(G)}^2(R)\ar[r] &\underset{|\alpha| = 2}{\bigoplus}
E_\alpha
\ar[r]&  \underset{|\alpha| = 3}{\bigoplus}
E_\alpha \ar[r]&  {\begin{array}{c}  E_{(1,1,0,1,1,0)}^2 \oplus E_{(1,0,1,1,0,1)}^2 \oplus E_{(0,1,1,0,1,1)}^2  \\ \oplus  \\ \left( \underset{\substack{|\alpha| = 4 \\ {\rm rest  \hskip 1mm of} \hskip 1mm \alpha's } }{\bigoplus}
E_\alpha \right) \end{array}} \ar[r] &} $}

\hskip 3cm {\small$ \xymatrix{ \ar[r] & {\begin{array}{c} \underset{|\alpha| = 5}{\bigoplus} E_\alpha \\ \oplus \\ E_{\bf 1} \end{array}}
\ar[r]&  E_{\bf 1} \ar[r]& 0},$}

{\small $ \xymatrix{ 0 \ar[r]& H_{J(G)}^3(R)\ar[r] &  E_{(1,1,0,1,1,0)} \oplus E_{(1,0,1,1,0,1)} \oplus E_{(0,1,1,0,1,1)}\ar[r]& E_{\bf 1} \ar[r]&
0},$}

\noindent where we are only considering those $\alpha$'s such that $\fp_\alpha \in \Supp  R/J(G)$. In this case we have that both injective resolutions have two linear strands.
\end{example}

\begin{example}
Let $G$ be a $6$-wheel where $x_6$ is the dominating vertex. We have that all the subgraphs $G_\alpha$ are connected except for the case $\alpha=(1,1,1,1,1,0)$ in which $G_\alpha$ is a $5$-cycle. Therefore we have the injective resolutions

\hskip -1cm{\small$ \xymatrix{ 0 \ar[r]& H_{J(G)}^2(R)\ar[r] &\underset{|\alpha| = 2}{\bigoplus}
E_\alpha
\ar[r]&  \underset{|\alpha| = 3}{\bigoplus}
E_\alpha \ar[r]&  {\begin{array}{c}  E_{(1,1,1,1,1,0)}  \\ \oplus  \\ \left( \underset{\substack{|\alpha| = 4   } }{\bigoplus}
E_\alpha \right) \end{array}} \ar[r] &  {\begin{array}{c}  E_{(1,1,1,1,1,0)}  \\ \oplus  \\ \left( \underset{\substack{|\alpha| = 5 \\ a_6=1 } }{\bigoplus}
E_\alpha \right) \end{array}} \ar[r] &  E_{\bf 1} \ar[r]& 0} $}

{\small $ \xymatrix{ 0 \ar[r]& H_{J(G)}^3(R)\ar[r] &  E_{(1,1,1,1,1,0)} \ar[r] \ar[r]&
0},$}

\noindent where we are only considering those $\alpha$'s such that $\fp_\alpha \in \Supp  R/J(G)$. The injective resolution of $H_{J(G)}^2$  has two linear strands.
\end{example}

More generally we can consider the following family of examples.

\begin{assumptions} \label{assump2}
Let $J(G) \subseteq R$ be the cover ideal of a simple (not necessarily connected) graph $G$ that is obtained by joining cycles and paths
in such a way that we can still find degree two vertices that we can remove in order to simplify the graph.  
\end{assumptions}

Even though it is not possible to give a closed formula for the Bass numbers
$\mu_p(\fp_\alpha, H^{k+1}_{J(G)}(R))$ as the one given in Theorem  \ref{Bass_trivial}, the methods developed in this work allow us to compute them. To do so we must performe the following steps:

\begin{itemize}
\item[$\cdot$] Describe the connected components of $G_\alpha$. 

\item[$\cdot$] Compute the Lyubeznik numbers of each component as in Section \ref{Lyubeznik}.

\item[$\cdot$]  Apply Proposition \ref{dual_join}.  
\end{itemize}

%

We can also discuss the vanishing of local cohomology modules depending on the connected components of the corresponding subgraphs but the results are not going to be as clean as in Proposition \ref{max}.  Indeed, using Proposition \ref{cycle} and Example \ref{cmax}, we deduce the following formula for the case of cycles.

\begin{proposition}
Let $G$ be a cycle of the form $C_{3k-1}$. Then $$H^{k+1}_{J(G)}(R)) \neq 0  \hskip 5mm {\rm for} \hskip 5mm  k=1, \dots , c_{max}+1.$$
On the other hand,  if $G$ is a cycle of the form $C_{3k}$ or $C_{3k+1}$ we have  $$H^{k+1}_{J(G)}(R)) \neq 0  \hskip 5mm {\rm for} \hskip 5mm  k=1, \dots , c_{max}.$$
\end{proposition}

Some partial results that we can provide are the following.

\begin{proposition} \label{vanishing_cycle}
Let $G$ be a graph in $n$ vertices and assume that, given $\alpha \in \{0,1\}^n$, the subgraph $G_\alpha$  has $r$ connected components 
having the Lyubeznik type of cycles $C_{n_1}, \dots , C_{n_r}$ and $s$ connected components having trivial Lyubeznik table. 
Assume that $n_j= 3k_j + \ell_j$ with $\ell_j\in \{-1,0,1\}$.
Then, if we denote $d_\alpha:=|\alpha|-2 = \dim(R_{\fp_\alpha}/J(G_\alpha))$, we have $$H^{d_\alpha - i}_{J(G_\alpha)}(R_{\fp_\alpha})) \neq 0  \hskip 5mm {\rm for} \hskip 5mm  i=0, \dots , (k_1 + \cdots + k_r + s) -1 .$$ In particular, the cohomological dimension of $J(G_\alpha)$ is 
$${\rm cd}(J(G_\alpha), R_{\fp_\alpha})= |\alpha| - (k_1 + \cdots + k_r + s) -1 .$$
\end{proposition}

\begin{proof}
Consider the decomposition $G_\alpha = G_1 \cup G_2$ where $\alpha=\alpha_1 + \alpha_2$ and  $G_1$ is the subgraph in $|\alpha_1|$ vertices containing the components corresponding to cycles and $G_2$ is the subgraph in $|\alpha_2|$ vertices containing the rest.  For $G_2$, as a consequence of Corollary \ref{disjoint_trivial}, we have
$$H^{d_2 - i}_{J(G_2)}(R_{\fp_{\alpha_2}})) \neq 0  \hskip 5mm {\rm for} \hskip 5mm  i=0, \dots , s -1 ,$$
where $ d_2= |\alpha_2|-2$.

On the other hand, using Corollary \ref{cycle} and   Corollary \ref{disjoint_top}, we have that the smallest $i$ for which the Lyubeznik number $\lambda_{p,i}$ corresponding to $G_1$ is non zero is
\begin{align*}
i& = (n_1-2)-(k_1-1) + \cdots +  (n_r-2)-(k_r-1) + (r-1) \\
& = (n_1+\cdots + n_r) -2 -2(r-1) -(k_1+\cdots + k_r)+ r + (r-1) \\
& = |\alpha_1| -2 -(k_1+\cdots + k_r) +1
\end{align*}
Moreover, since the non-vanishing local cohomology modules of a cycle are consecutive, that is $H^k_{J(C_n)}(R)\neq 0$ where $k$  runs from two to the cohomological dimension, it follows from Proposition \ref{dual_join} that the same consecutiveness property holds for $G_1$. Namely we have
$$H^{d_1 - i}_{J(G_1)}(R_{\fp_{\alpha_1}}) \neq 0  \hskip 5mm {\rm for} \hskip 5mm  i=0, \dots , (k_1+\cdots + k_r)  -1 ,$$
where $ d_1= |\alpha_1|-2$.

Finally, applying Corollary \ref{disjoint_top} for $G_1$ and $G_2$, we have that the smallest $i$ for which  $\lambda_{p,i}(R_{\fp_\alpha}/J(G_\alpha))\neq 0$ is
\begin{align*}
i&  =( |\alpha_1| -2 -(k_1+\cdots + k_r) +1 ) + (|\alpha_2|-2 - (s-1))+1 \\
&=   |\alpha_1| + |\alpha_2| -2 - (k_1+\cdots + k_r + s)  +1 \\
&= |\alpha| -2   - (k_1+\cdots + k_r + s)  +1 \\
&= d_\alpha   - (k_1+\cdots + k_r + s)  +1
\end{align*}
Once again Proposition \ref{dual_join} gives the consecutiveness of the non-vanishing local cohomology modules and the result follows.
\end{proof}

More generally, and using the same type of arguments as above, we have.

\begin{proposition}\label{vanishing_cycle2}
Let $G$ be a graph in $n$ vertices and assume that, given $\alpha \in \{0,1\}^n$, the subgraph $G_\alpha$  has $r$ connected components 
having the Lyubeznik type of cycles $C_{n_{1,i}}\divideontimes \cdots \divideontimes C_{n_{t_i,i}}$, for $i=1,\dots , r$ and $s$ connected components having trivial Lyubeznik table. 
Assume that $n_{j,i}= 3k_{j,i} + \ell_{j,i}$ with $\ell_{j,i}\in \{-1,0,1\}$.
Then, if we denote $d_\alpha:=|\alpha|-2 = \dim(R_{\fp_\alpha}/J(G_\alpha))$, we have $$H^{d_\alpha - i}_{J(G_\alpha)}(R_{\fp_\alpha})) \neq 0  \hskip 3mm {\rm for} \hskip 3mm  i=0, \dots , (k_{1,1}+\cdots + k_{t_1,1} + \cdots + k_{1,r}+\cdots + k_{t_r,r} + s)-(t_1+\cdots + t_r) +(r-1) .$$ 

\end{proposition}

\end{document}